\documentclass[12pt]{article}
\usepackage{amsmath}
\usepackage{amsfonts}
\usepackage{amsthm}
\usepackage{graphicx}
\graphicspath{{Figures/}} 
\usepackage{overpic}
\usepackage{amssymb} 
\usepackage{pictex}
\usepackage{rotating}  
\usepackage{xcolor}
\usepackage{cite} 

\usepackage{enumitem} 

\usepackage{accents} 
\usepackage{float}
\usepackage{tikz}

\definecolor{refkey}{rgb}{0,0,1}
\definecolor{labelkey}{rgb}{0,0,1}
\usepackage{comment}
\numberwithin{equation}{section}

\usepackage{etoolbox}

\newtheorem{theorem}{Theorem}[section]

\newtheorem{proposition}[theorem]{Proposition}
\newtheorem{lemma}[theorem]{Lemma}
\newtheorem{corollary}[theorem]{Corollary}
\newtheorem{Definition}[theorem]{Definition}
\newenvironment{definition}{\begin{Definition}\rm}{\end{Definition}}
\newtheorem{Remark}[theorem]{Remark}
\newenvironment{remark}{\begin{Remark}\rm}{\end{Remark}}
\newtheorem{Remarks}[theorem]{Remarks}
\newenvironment{remarks}{\begin{Remarks}\rm}{\end{Remarks}}
\newtheorem{Conjecture}[theorem]{Conjecture}
\newtheorem{Example}[theorem]{Example}

\newcommand{\B}{{\mathbb B}}
\newcommand{\C}{{\mathbb C}}
\newcommand{\D}{{\mathbb D}}

\newcommand{\R}{{\mathbb R}}

\newcommand{\gam}{\gamma}
\newcommand{\om}{\omega}

\newcommand{\p}{\partial}

\def \capa{{\rm cap}\,}

\def\det{\mathop{\mathrm{det}}\nolimits}

\def\supp{\mathop{\mathrm{supp}}\nolimits}

\renewcommand{\bar}{\overline}
\renewcommand{\tilde}{\widetilde}
\renewcommand{\hat}{\widehat}

\usepackage[bookmarksopen, naturalnames]{hyperref}
\begin{document}
\title{Riesz equilibrium on a ball 
in the external field of a point charge}
\author{P.D. Dragnev, R. Orive, E.B. Saff and F. Wielonsky}

\maketitle

\begin{abstract}
We investigate the Riesz energy minimization problem on a $d$-dimensional ball in the presence of an external field created by a point charge above the ball in $\R^{d+1}$, $d\geq1$. Both cases of an attractive charge and a repulsive charge are considered. The notion of a signed equilibrium measure is one of the main tools in the present study.
For the case of a positive (repulsive) charge, the determination of the support of the equilibrium measure is a nontrivial question. We solve it in the one-dimensional case by making use of iterated balayage, a method already applied in logarithmic potential theory. Here we use a modified version of it, in order to handle the phenomenon of mass loss, characteristic of the Riesz balayage of positive measures. Moreover, we also consider minimization of Coulomb energy on the ball in dimension $d\geq2$, and of logarithmic energy on the segment in dimension 1. Different techniques are used for these two cases.
\end{abstract}
\section{Introduction}
The weighted Riesz equilibrium problem consists in minimizing the weighted energy 
\begin{equation}\label{def-I-U}
I_{Q}(\mu)=I(\mu)+2\int Q(x)d\mu(x),\quad I(\mu)=\iint\frac{d\mu(x)d\mu(t)}{|x-t|^{s}},
\qquad0<s<d,
\end{equation}
of a probability measure $\mu$ supported on a given closed set $\Sigma\subset\R^{d}$, of positive capacity, in the presence of an external field $Q(x)$ defined on $\Sigma$.
The potential of $\mu$, and the unweighted energy $I(\mu)$ in (\ref{def-I-U}) can be written as
$$
U^{\mu}(x)=\int\frac{d\mu(t)}{|x-t|^{s}},
\qquad I(\mu)=\int U^{\mu}(x)d\mu(x).
$$
It follows from the theory of Riesz potentials that, when $0<s<d$ and under suitable assumptions on $Q$, the equilibrium measure $\om_{Q,\Sigma}$ minimizing the energy $I_{Q}(\mu)$ exists and is unique.

In the present paper, $\Sigma$ is chosen to be the closure $\bar\B$ of the $d$-dimensional unit ball $\B\subset\R^{d}$, and the external field $Q$ is created by a charge $\gam$ at a position $y=(0,y_{d+1})\in\R^{d+1}$, $y_{d+1}\geq0$,  above $\B$,
\begin{equation}\label{ext-field}
Q(x)=\frac{\gamma}{|x-y|^{s}}=\frac{\gam}{(|x|^{2}+y_{d+1}^{2})^{s/2}}.
\end{equation}
For ease of notation, since we are only considering the ball of dimension $d$ and the sphere of dimension $d-1$ (except at one occurrence), we agree to denote the unit ball and the unit sphere respectively by $\B$ and $\mathbb{S}$, instead of $\B^{d}$ and $\mathbb{S}^{d-1}$.

Concerning the Riesz parameter $s$, we consider the Robin case 
\begin{equation}\label{s-value}
d-2<s<d,~\text{if }d\geq2,\quad\text{and}\quad0<s<1,~\text{if }d=1,
\end{equation}
also called the subharmonic case
since a potential $U^{\mu}$ is subharmonic outside of the support of $\mu$ when (\ref{s-value}) is satisfied. When $d\geq3$, we also study Coulomb interaction ($s=d-2$), and when $d=1,2$, we  consider the case $s=0$, that is logarithmic interaction. It will be convenient to introduce the parameter $0<\alpha<d$ such that 
$$\alpha=d-s.$$
There have been many investigations about Riesz equilibria on a sphere in $\R^{d}$, see e.g.\ \cite{BDS2009} and \cite{BDS14}, among others, but the case of a ball seems to be less studied.
One of the main difficulties in the present work is the fact that the external field $Q(x)$ in (\ref{ext-field}) is not a convex function of $|x|\in[0,\infty)$, though it is convex on $[y_{d+1}/\sqrt{s+1},\infty)$ when $\gamma>0$, and on $[0,y_{d+1}/\sqrt{s+1}]$ when $\gamma<0$. This prevents us from applying known results about the support of the equilibrium measure which hold in case of a convex external field, see e.g. \cite{ST} for the case of the logarithmic kernel, and \cite{B} for the case of Riesz kernel.

Riesz equilibrium problems on a ball in $\R^{d}$, in the presence of an external field $Q$, were previously studied in \cite{B}. In particular, for
the case of a positive charge placed at some height $h$ above the ball, it is shown that
if $h>h_{0}$ where $h_{0}$ is some specific number,
then the support of the equilibrium $\mu_{Q}$ is the full ball, and its density is computed, see \cite[Theorem II.6]{B}.
In the Coulomb case $d=3$, $s=1$, the critical height when the support starts to have an opening at the origin is computed, see \cite[Corollary II.8]{B}.
If $Q$ is a general radial field, and it is assumed that the support is a shell, a way to reconstruct the equilibrium on that shell is given, in terms of a complicated Fredholm equation, see \cite[Theorem II.9]{B}.
We also mention that the case of an attractor located above the whole space $\R^{d}$ was investigated in \cite{DOSW}.

Here, we aim at studying how the weighted equilibrium problem on $\bar\B$ evolves as the strength $\gamma$ of the charge moves from $-\infty$ to $+\infty$; that is, from a very attractive charge to a very repulsive charge. 
The paper is organized as follows.
In Section \ref{Prelim} we display some preliminary results useful for the sequel. In particular, we show that the balayage of a measure with compact support onto a disjoint compact set is absolutely continuous with respect to Lebesgue measure, with a real analytic density, a result which may be of independent interest (see Proposition \ref{bal-abs-cont}).
In Section \ref{Sec-Riesz} we first consider the case of Riesz interaction, with an attractive charge, whose study is essentially based on the approach developed in \cite{DOSW}. Then we study
the more delicate case of a repulsive charge, and we propose the {\it shell conjecture}, namely that the support of the equilibrium measure is either the entire ball $\bar\B$ or a shell with outer boundary the $(d-1)$-dimensional unit sphere $\mathbb{S}$.
In Section \ref{Coulomb}, we consider Coulomb interactions, i.e.\ the case $s=d-2$, including logarithmic interaction for the disk.
In Section \ref{log-seg}, we consider the problem for the segment with logarithmic interaction.
Finally, in Section \ref{shell}, which is the most technical part of the paper, we prove the above mentioned shell conjecture in the case of a segment.
For that, we adapt to the Riesz setting the method of {\it iterative balayage} that was used in \cite{KD}, see also \cite{DK}, for the study of a weighted  logarithmic equilibrium problem on an interval on the real line. See Table 1 for the locations of the special cases considered.
\begin{table}[H]
\begin{center}
\begin{tabular}{|c|c|}
\hline
\rule{0pt}{3ex}
$d\geq2$ & $d=1$ \\
\hline
\rule{0pt}{3ex} 
$d-2<s<d$ : Section \ref{Sec-Riesz} & $0<s<1$ : Sections \ref{Sec-Riesz} and \ref{shell}
\rule[-1.2ex]{0pt}{0pt} 
\\
\hline
\rule{0pt}{3ex}
$s=d-2$ : Section \ref{Coulomb} & $s=0$ : Section \ref{log-seg}
\rule[-1.2ex]{0pt}{0pt}\\
\hline
\end{tabular}
\caption{The different cases studied and the corresponding sections}
\end{center}
\end{table}
\section{Preliminaries}\label{Prelim}
Hereafter, all the positive measures are locally finite. For a signed measure $\nu$ in $\R^{d}$, we denote by $S_{\nu}$, or $\supp\nu$, its support. We denote its signed total mass by
$$m(\nu):=\nu(\R^{d})=\nu(\supp\nu),$$
or also $\|\nu\|$ when $\nu$ is a positive measure.

We denote by $\nu^{+}$ and $\nu^{-}$ the positive and negative parts in the Jordan decomposition of $\nu=\nu^{+}-\nu^{-}$. 
We also set $U^{\nu}=U^{\nu^{+}}-U^{\nu^{-}}$ and 
$$I(\nu)=I(\nu^{+})-2I(\nu^{+},\nu^{-})+I(\nu^{-}),\qquad I(\nu^{+},\nu^{-})
=\int\frac{d\nu^{+}(x)d\nu^{-}(t)}{|x-t|^{s}},
$$
where it is assumed that the right-hand side of the first equality is well-defined. 
We recall that a signed measure $\nu$ has finite energy if and only if both $\nu^+$ and $\nu^-$ have finite energy
(see \cite{Fu}, as well as \cite[Definition 4.2.4, p.134]{BHS} for more general kernels). 

We define an admissible external field $Q$ on a compact set $K$ as a lower semicontinuous function  $Q:K\to(-\infty,+\infty]$, such that the Riesz capacity of the set $\{z\in K,~Q(z)<\infty\}$ is positive.
The weighted equilibrium measure $\om_{Q,K}$ of $K$ and $Q$, supported on $K$, which minimizes the energy $I_{Q,K}$, exists and is uniquely characterized by the Frostman, or variational inequalities
\begin{align}
U^{\om_{Q,K}}(x)+Q(x) & \geq F_{Q},\quad\text{q.e. }x\in K, \label{Frostman1}
\\[5pt]
U^{\om_{Q,K}}(x)+Q(x) & \leq F_{Q},\quad x\in S_{Q,K}, \label{Frostman2}
\end{align}
where $F_{Q}$ is referred to as the equilibrium constant, and where we simply denote by $S_{Q,K}$ the support of $\om_{Q,K}$. Note that, integrating the equality $U^{\om_{Q,K}}(x)+Q(x)=F_{Q}$, q.e.\ on $S_{Q,K}$, with respect to its unweighted equilibrium measure $\om_{S_{Q,K}}$, we get that
\begin{equation}\label{F-value}
F_{Q}=\int U^{\om_{S_{Q,K}}}d\om_{Q,K}+\int Q(x)d\om_{S_{Q,K}}(x)=
I_{S_{Q,K}}+\int Q(x)d\om_{S_{Q,K}}(x),
\end{equation}
where we have used that $U^{\om_{S_{Q,K}}}=I_{S_{Q,K}}$ q.e.\ on $S_{Q,K}$.

We first state a basic result about the unweighted equilibrium measure $\om_{K}$ on a compact set $K$, recall that (\ref{s-value}) is assumed to hold throughout.
\begin{theorem}[{\cite[Theorems 1-4]{W}}]\label{Wal}
Let $K\subset\R^{d}$ be a compact set of positive capacity. 
\\[5pt]
i) The support of the equilibrium measure
$\om_{K}$ contains the interior $\accentset{\circ}{K}$ of $K$ (and thus also its closure).
\\[5pt]
ii) The restriction of $\om_{K}$ to the interior of $K$ is absolutely continuous with respect to the Lebesgue measure, and has a real analytic density, given by
\begin{equation}\label{dens-mes-eq}
\om_{K}'(x)=A(d,\alpha)\int\frac{F_{K}-U^{\om_{K}}(y)}{|x-y|^{d+\alpha}}dy,\qquad x\in\accentset{\circ}{K},
\end{equation}
where $F_{K}$ is the unweighted equilibrium constant of $K$, and $A(d,\alpha)$ is a constant dependent of $d$ and $\alpha$.
\\[5pt] 
iii) Let $x_{0}$ be a point on the boundary of $K$, and assume that in a neighborhood of $x_{0}$ in $\p K$, some inner and outer ball conditions are satisfied (see \cite[Theorem 4]{W} for details). Then the density $\om'_{K}(x)$ behaves like $M|x-x_{0}|^{-\alpha/2}$ as $x\to x_{0}$, where $M$ is some positive constant (depending on $x_{0}$).
\end{theorem}
Explicit expression for the density with respect to Lebesgue measure of the equilibrium measure of the closed ball $B_R$ of radius $R$ in $\R^{d}$, simply denoted $\omega_R$, can be found in \cite[p.\ 163]{Land}, \cite[Eq.(4.6.12)]{BHS}, namely
\begin{equation}\label{equilball}
\omega_R' (x) =  \frac{c_{R}}{(R^2 - |x|^2)^{\alpha/2}},\qquad
c_{R}= \frac{\pi^{-d/2}\Gamma (1+s/2)}{R^{s}\Gamma (1-\alpha/2)},
\end{equation}
Let us now recall the notion of {\em balayage} of a measure when $d-2 \leq s <d$, see \cite[Chapter IV, \S 5, p.264]{Land}.
Given
a compact set $K\subset \R^d$ of positive capacity and a {positive} measure $\sigma$ \textcolor{black}{of finite total mass}, 
there exists a unique {positive} measure
$$\widehat{\sigma}:=
Bal(\sigma,K)$$
called the Riesz $s$-balayage of $\sigma$ onto $K$ satisfying the following: \\
1) $S_{\widehat{\sigma}} \subseteq K$, \\
2) $\widehat{\sigma}$ is zero on the set of irregular points of \textcolor{black}{the complement of $K$}, and
\begin{equation}\label{balayage}
U^{\widehat{\sigma}}(x) = U^{\sigma}(x)\text{ q.e.\ on }K,\qquad U^{\widehat{\sigma}}(x) \leq U^{\sigma}(x)\text{ on }\R^d.
\end{equation}
\begin{remark}\label{Rem-mass-loss}
The mass of $\hat\sigma$ satisfies $\|\hat\sigma\|\leq\|\sigma\|$, i.e.\ the balayage may induce a mass loss 
and the energy of $\hat\sigma$ satisfies $I(\hat\sigma)\leq I(\sigma)$. Actually, the following formula holds, see \cite[p.264]{Land},
\begin{equation}\label{mass-bal}
\|\hat\sigma\|=\capa(K)\int U^{\sigma}d\om_{K}=\capa(K)\int U^{\om_{K}}d\sigma.
\end{equation}
\end{remark}
It may also be useful to note that, for two compact sets $K_{1}\subset K_{2}$, one has
$$
Bal(Bal(\sigma,K_{2}),K_{1})=Bal(\sigma,K_{1}).
$$
For a signed measure $\nu=\nu^{+}-\nu^{-}$, we define the balayage of $\nu$ on a compact set $K$ by $$\hat\nu=\hat{\nu^{+}}-\hat{\nu^{-}}.$$
Let us now state a few known results about the balayage of point masses.
\begin{proposition}[{\cite[p.121-122]{Land}}] The balayages of a point mass $\delta_{t}$, $t\in\R^{d}$,  onto the outside (for $|t|<r$), and inside (for $|t|>r$) of the ball $B_{r}\subset\R^{d}$ , are respectively given, up to a multiplicative constant, by
\begin{equation}\label{Bal-in-out}
\hat\delta_{t}(x)=\left(\frac{r^{2}-|t|^{2}}{|x|^{2}-r^{2}}\right)^{\alpha/2}
\frac{dx}{|t-x|^{d}},\qquad
\hat\delta_{t}(x)=\left(\frac{|t|^{2}-r^{2}}{r^{2}-|x|^{2}}\right)^{\alpha/2}
\frac{dx}{|t-x|^{d}}.
\end{equation}
\end{proposition}
In the next lemma, we consider the Kelvin transform $K_{a}:\R^{d}\to\R^{d}$, with center $a\in\R^{d}$ and radius $1$. Recall it is defined by the relation
$$
K_{a}(x)-a=\frac{x-a}{|x-a|^{2}},\qquad x\in\R^{d}.
$$
The Kelvin transform is involutive, its Jacobian $J_{a}$ satisfies
$$|\det J_{a}(t)|=|t-a|^{-2d}.$$
There is a known connection between balayages and equilibrium measures, see \cite[Chapter IV, p.263]{Land}.
\begin{lemma}(Balayage and equilibrium measure) \label{Bal-eq}
Let $K$ be a compact subset of $\R^{d}$, and let
$K_{a}$ be the Kelvin transform with center $a\notin K$ and radius $1$. Then, 
\begin{equation}\label{Bal=eq}
Bal(\delta_{a},K)(t)=F_{K_{a}(K)}^{-1}|t-a|^{s}d\om_{K_{a}(K)}(K_{a}(t)),\qquad t\in K,
\end{equation}
where $F_{K_{a}(K)}$ denotes the unweighted equilibrium constant for the compact set $K_{a}(K)$.
\end{lemma}
For completeness, we give a proof.
\begin{proof}
The equilibrium measure $\om_{K_{a}(K)}$ satisfies, 
$$
U^{\om_{K_{a}(K)}}(x)=F_{K_{a}(K)},\qquad x\in K_{a}(K).
$$
Let $\mu_{K}$ be the measure on $K$ defined by $d\mu_{K}(t)=|t-a|^{s}d\om_{K_{a}(K)}(K_{a}(t))$. Then, for $x\in K$,
\begin{align*}
U^{\mu_{K}}(x) & =\int_{K}\frac{|t-a|^{s}}{|x-t|^{s}}d\om_{K_{a}(K)}(K_{a}(t))
=|x-a|^{-s}U^{\om_{K_{a}(K)}}(K_{a}(x))
\\[5pt]
& =F_{K_{a}(K)} {|x-a|^{-s}}
=F_{K_{a}(K)}U^{\delta_{a}}(x),
\end{align*}
where, in the second equality, we use the fact that
\begin{equation}\label{Kel2}
|x-a||t-a||K_{a}(x)-K_{a}(t)|=|x-t|.
\end{equation}
The equality between potentials shows the desired result.
\end{proof}
\begin{lemma}[Superposition principle, {\cite[p.264]{Land}}]
For a signed measure $\mu$ and a nonnegative measurable function $f$ on K, the following formula holds:
\begin{equation}\label{super}
\int_{K} fd\hat\mu=\int_{S_{\mu}}\left(\int_{K} fd\hat\delta_{t}\right)d\mu(t).
\end{equation}
\end{lemma}
From Theorem \ref{Wal} and the two preceeding lemmas, we obtain (assuming as above $d\geq1$ and $\max(0,d-2)<s<d$)
\begin{proposition}\label{bal-abs-cont}
Let $K$ be a compact subset of $\R^{d}$ and $a\notin K$. Then
\\[5pt]
1) The restriction of $Bal(\delta_{a},K)$ to the interior of $K$ is absolutely continuous with respect to the Lebesgue measure. Its density is a real analytic function on $\accentset{\circ}{K}$. 
\\[5pt]
2) Let $\mu$ be a positive measure whose support is compact and disjoint from $K$. The restriction of the balayage $Bal(\mu,K)$ to the interior of $K$ is absolutely continuous, and has a real analytic density.
\\[5pt]
3) The behavior of the density of $Bal(\delta_{a},K)$ near a boundary point $x_{0}$ of $K$ which satisfies the geometric assumptions stated in Theorem \ref{Wal} (iii), behaves like $M|x-x_{0}|^{-\alpha/2}$ as $x\to x_{0}$, where $M$ is some positive constant. 
\end{proposition}
\begin{proof}
1) Let $E$ be a Lebesgue null set in $\accentset{\circ}{K}$. Since $K_{a}$ is differentiable, $K_{a}(E)$ is also a null set (cf.\ \cite[Lemma 7.25]{Rud}). Hence $\om_{K_{a}(K)}(K_{a}(E))=0$, which, in view of (\ref{Bal=eq}), implies $Bal(\delta_{a},K)(E)=0$. Moreover, with $\om_{K_{a}(K)}'$ the real analytic density of $\om_{K_{a}(K)}$, we get from (\ref{Bal=eq}) that the density of $Bal(\delta_{a},K)$ equals, 
\begin{equation}\label{exp-dens-bal}
\hat\delta_{a}'(t)=F_{K_{a}(K)}^{-1}|t-a|^{s}\om_{K_{a}(K)}'(K_{a}(t))|\det J_{a}(t)|=
F_{K_{a}(K)}^{-1}|t-a|^{s-2d}\om_{K_{a}(K)}'(K_{a}(t)),
\end{equation}
which is a real analytic function of $t$. Indeed,
$$
|t-a|^{s-2d}=\left(\sum_{j=1}^{d}(t_{j}-a_{j})^{2}\right)^{s/2-d}=\exp\left((s/2-d)\log\sum_{j=1}^{d}(t_{j}-a_{j})^{2}\right)
$$ 
is real analytic on $\R^{d}\setminus\{a\}$, and $K_{a}(t)$ is also real analytic on $\R^{d}\setminus\{a\}$.
\\
2) If $E$ is a subset of the interior of $K$ of Lebesgue measure 0, it is clear that $\hat\mu(E)=0$ from (\ref{super}) with $f$ the characteristic function of $E$ and the fact that, for any $t$, $\hat\delta_{t}(E)=0$. It follows from (\ref{dens-mes-eq}), (\ref{super}) and (\ref{exp-dens-bal}) that the density of $\hat\mu$ is given by
\begin{align*}
\hat\mu'(t) & =\int_{a}\hat\delta_{a}'(t)d\mu(a)
=\int_{a} F_{K_{a}(K)}^{-1}|t-a|^{s-2d}\om_{K_{a}(K)}'(K_{a}(t))d\mu(a)
\\[5pt]
& = A(d,\alpha)\int_{a} F_{K_{a}(K)}^{-1}|t-a|^{s-2d}
\int_{y}\frac{F_{K_{a}(K)}-U^{\om_{K_{a}(K)}}(y)}{|K_{a}(t)-y|^{d+\alpha}}dyd\mu(a)
\end{align*}
With the change of variable $y=K_{a}(u)$ and making use of (\ref{Kel2}), we get
\begin{equation}\label{dens-bal-mu}
\hat\mu'(t) =A(d,\alpha)\int_{a\in\supp\mu} F_{K_{a}(K)}^{-1}
\int_{u\in\R^{d}\setminus K}\frac{F_{K_{a}(K)}-U^{\om_{K_{a}(K)}}(K_{a}(u))}{|t-u|^{d+\alpha}|u-a|^{d-\alpha}}dud\mu(a).
\end{equation}
Note that, for $t\in\accentset{\circ}{K}$,
$$
\hat\mu'(t)\leq A(d,\alpha)\int_{a\in\supp\mu} 
\int_{u\in\R^{d}\setminus K}\frac{du}{|t-u|^{d+\alpha}|u-a|^{d-\alpha}}d\mu(a)<+\infty,
$$
Indeed, it follows from 
splitting the inner integral over $|u-a|\leq1$ and $|u-a|>1$, and using the fact that $\mu$ has a compact support. Now, for a given $t_{0}\in\accentset{\circ}{K}$, $|t-u|^{-d-\alpha}$ can be written as a power series centered at $t_{0}$, in a neighborhood of $t_{0}$, independent of $u\in\R^{d}\setminus K$. This, together with (\ref{dens-bal-mu}) and Tonelli's theorem, shows that $\hat\mu'(t)$ is real analytic when $t\in\accentset{\circ}{K}$.

3) This is a consequence of (\ref{exp-dens-bal}), (\ref{Kel2}), and the fact that the geometric assumptions in iii) of Theorem \ref{Wal}, are invariant under a Kelvin transform.
\end{proof}
\begin{lemma}\label{lem-iba}
Assume $(K_{n})_{n\geq0}$ is a sequence of decreasing compact sets with $K=\cap_{n} K_{n}$, with $K$ of positive capacity. Let $\nu$ be a signed measure, with $\nu^{+}$ and $\nu^{-}$ of finite energy. Then, as $n\to\infty$,
$$
Bal(\nu,K_{n})\to Bal(\nu,K)\quad\text{weak-*}.
$$
\end{lemma}
\begin{proof}
First consider the case of $\nu$ a positive measure of finite energy. 
Assume $\tilde\nu$ is a weak-* limit of some subsequence of $Bal(\nu,K_{n})$. Then, by definition of balayage,
$$
\forall n,\qquad U^{Bal(\nu,K_{n})}=U^{Bal(\nu,K)},  \qquad\text{ q.e.\ on }K,
$$
and
$$
U^{\tilde\nu}=\liminf_{n}U^{Bal(\nu,K_{n})}=U^{Bal(\nu,K)},  \qquad\text{ q.e.\ on }K,
$$
where we use the lower envelope theorem for the first equality. Since 
$Bal(\nu,K)$ is of finite energy, and
$$
I(\tilde\nu)\leq\liminf_{n}I(Bal(\nu,K_{n}))\leq I(\nu)<\infty,
$$
we derive from the uniqueness theorem, see \cite[p.178]{Land}, that $\tilde\nu=Bal(\nu,K)$. Hence, the entire sequence $Bal(\nu,K_{n})$ tends weak-* to $Bal(\nu,K)$.

If $\nu$ is a signed measure with $I(\nu^{+}), I(\nu^{-})<\infty$, by applying the previous result, we get, as $n\to\infty$,
$$
Bal(\nu,K_{n})=Bal(\nu^{+},K_{n})-Bal(\nu^{-},K_{n})\to Bal(\nu^{+},K)-Bal(\nu^{-},K)
=Bal(\nu,K),
$$
where the convergence is in the weak-* sense.
\end{proof}
In the sequel, we will also use the notion of signed equilibrium measure.
\begin{definition}
Let $\Sigma$ be a closed subset of $\R^d$. A {\it  signed equilibrium measure} for $\Sigma$ in the external field $Q$ is a {(finite)} signed measure $\eta_{Q,\Sigma}$ {with finite energy}, supported on $\Sigma$, such that $m(\eta_{Q,\Sigma}) = 1$, and there exists a finite constant $C_{Q,\Sigma}$ such that
\begin{equation}\label{defsigned}
U^{\eta_{Q,\Sigma}}(x) + Q(x) = C_{Q,\Sigma}\quad\text{q.e. on }\Sigma.
\end{equation}
\end{definition}
If this signed equilibrium measure exists, which we assume in the sequel, then it is unique, see \cite[Lemma 23]{BDS2009}.
It follows from our definition that both $\eta_{Q,\Sigma}^+$ and $\eta_{Q,\Sigma}^-$ have finite energy.
The main properties of $\eta_{Q,\Sigma}$ are the following ones, see \cite[Lemma 3.15]{DOSW},
\begin{lemma}\label{inclusion}
Let $\om_{Q,\Sigma}$ be the equilibrium measure for $\Sigma$ and $Q$, which we assume to exist.
\\[5pt]
(i) One has
\begin{equation}\label{signeddomin}
\om_{Q,\Sigma} \leq \eta_{Q,\Sigma}^+.
\end{equation}
In particular,
$S_{{Q,\Sigma}} \subseteq S_{\eta_{Q,\Sigma}^+}$.
\\[5pt]
(ii) Let $\Sigma_{1}$ be a closed subset of $\Sigma$ that admits a signed equilibrium measure $\eta_{Q,\Sigma_1}$, and such that $S_{{Q,\Sigma}}\subset\Sigma_{1}$. If $\eta_{Q,\Sigma_{1}}$ is a positive measure, then $\om_{Q,\Sigma}=\eta_{Q,\Sigma_{1}}$.
\end{lemma}
We next recall \cite[Lemma 5.5]{DOSW} which describes the balayage of the point mass $\delta_{y}$, $y=(0,y_{d+1})$, onto the ball $B_{R}\subset\R^{d}$ of radius $R$. It follows from 3) of Proposition \ref{bal-abs-cont} that its density near the boundary of $B_{R}$ behaves, up to a multiplicative constant, like $(R^{2}-|x|^{2})^{-\alpha/2}$. Hence, we set
$$
\Lambda_{R}^{*}:= \lim_{|x|\rightarrow R_{-}}(R^2 - |x|^2)^{\alpha/2}Bal'(\delta_{y},B_{R})(x).
$$
We denote by $A_{d}$ the surface area of the $d$-dimensional sphere $\mathbb{S}^{d}$, and by $W(\mathbb{S}^{d})$ its Riesz energy, with values
\begin{align*}
A_{d}=\frac{2 {\pi}^{(d+1)/2}}{\Gamma \left({(d+1)}/{2}\right)}, & \qquad 
W(\mathbb{S}^{d})= \begin{cases} {\displaystyle \frac{\Gamma \left(\frac{d+1}{2}\right)\Gamma (\alpha)}{\Gamma \left(\frac{\alpha+1}{2}\right)\Gamma \left(d-\,\frac{s}{2} \right)}}, & 0<s<d,\,d\geq 3,
\\[10pt]
{\displaystyle 2^{1-s}/(2-s)}, & 0<s<2, \,d=2, \end{cases}
\\[10pt]
W(\mathbb{S}^{1}) & =2^{-s}\pi^{-1/2}\Gamma((1-s)/2)/\Gamma(1-s/2),\quad0<s<1,
\end{align*}
see \cite[Formula (4.6.5)]{BHS} for $W(\mathbb{S}^{d})$.
\begin{lemma}\label{lem:balR}
Let $y = (0; y_{d+1})$ with $y_{d+1} >0$. 
\\[5pt]
(i) The density $Bal' (\delta_y, B_{R})(x)$, $|x|<R$,
is given by
\begin{equation}\label{balayageBR}
Bal' (\delta_y, B_R)(x) = \frac{(2y_{d+1})^{\alpha}}{W(\mathbb{S}^{d})A_{d}}
 \left(\frac{1}{(|x|^2 + y_{d+1}^2)^{d-s/2}}+
 \frac{\sin(\alpha\frac{\pi}{2})I(x)}{\pi (R^2 - |x|^2)^{\alpha/2}}\right),
\end{equation}
where 
\begin{equation*}
I(x)=\int_0^{+\infty}\frac{v^{\alpha/2}\,dv}{(v+R^2+y^2_{d+1})^{d-s/2} (v+R^2-|x|^2)}.
\end{equation*}
(ii) We have 
\begin{equation}\label{beh-bal}
\Lambda_{R}^{*}=
\frac{y_{d+1}^{\alpha}K_{s,d}^{(1)}}{(R^2 + y_{d+1}^2)^{d/2}},\qquad
K_{s,d}^{(1)}=\frac{2^{\alpha}\sin(\alpha\frac{\pi}{2})B(d/2,\alpha/2)}
{\pi A_{d}W(\mathbb{S}^{d})}.
\end{equation}
\end{lemma}
Now, the signed equilibrium measure, denoted $\eta_{Q,R}$, 
for the closed ball $B_{R}$ exists and can be expressed in terms of the balayage $Bal (\delta_y,B_{R})$ and the equilibrium measure $\omega_{R}$ as 
\begin{equation}\label{sgn-meas}
\eta_{Q,R} = - \gamma Bal (\delta_y,B_{R}) + (1 + \gamma m_R)\omega_R,
\end{equation}
where $m_R$ denotes the mass of $Bal (\delta_y, B_R)$. 
The density of $\eta_{Q,R}$ near the boundary $|x| = R$ satisfies 
\begin{align}
H(R) & := \lim_{|x|\rightarrow R}(R^2 - |x|^2)^{\alpha/2}\eta'_{Q,R}(x) =
-\gamma\Lambda_{R}^{*}+(1+\gamma m_{R})c_{R} \notag
\\
& =-\frac{\gamma y_{d+1}^{\alpha}K_{s,d}^{(1)}}{(R^2 + y_{d+1}^2)^{d/2}}
+\frac{\Gamma (1+s/2)}{\pi^{d/2} \Gamma (1-\alpha/2)}\frac{(1+\gamma m_{R})}{R^{s}}. \label{HR}
\end{align}
For $\gamma<-1$, the equation $H(R)=0$ has a unique solution $R=y_{d+1}\sqrt{z}$, where $z$ is the unique positive solution of the equation
\begin{equation}\label{sol}
z^{s/2+1}{}_2F_{1}\left(1+\frac{s}{2},1+\frac{d}{2},2+\frac{s}{2},-z\right)
= -\frac{\Gamma(\alpha/2)\Gamma(2+s/2)}{\gamma\Gamma(1+d/2)},
\end{equation}
see \cite[Theorem 4.1 (iii)]{DOSW}.
\section{The case $\max(0,d-2)<s<d$ for a ball in $\R^{d}$, $d\geq1$}
\label{Sec-Riesz}
\subsection{The case of an attractive charge $\gam\leq0$}
\label{Neg}
In this section, we study the Riesz equilibrium on the ball $\bar\B\subset\R^{d}$ with the external field $Q$ given by an attractive charge $\gamma$ at a fixed position $(0,y_{d+1})$, as $\gamma$ goes from $-\infty$ to 0.
\begin{lemma}\label{lem-neg}
(i) If $\gamma<-1$, there exists a unique $R_{\gamma}>0$ such that $H(R_{\gamma})=0$ and
$$
H(R)>0\text{ for }0<R<R_{\gamma},\qquad 
H(R)<0\text{ for }R>R_{\gamma},
$$
and
$$
\forall R\leq R_{\gamma},\quad\eta_{Q,R}=\eta_{Q,B_{R}}\text{ is a positive measure}.
$$
(ii) If $-1\leq\gamma<0$, we have
$$
\forall R>0,\quad H(R)>0\quad\text{ and }\quad
\eta_{Q,R}\text{ is a positive measure}.
$$
\end{lemma}
\begin{proof}
(i) was derived in the proof of \cite[Theorem 4.1]{DOSW}, except for the fact that, for $R<R_{\gamma}$, $\eta_{Q,R}$ is a positive measure. This fact can be proved in the same way we proved that $\eta_{Q,R_{\gamma}}$ is a positive measure, see the proof of \cite[Theorem 4.1]{DOSW}, where equation (5.21) there has to be replaced with
\begin{equation*}
H(R)=-\gamma\Lambda_{R}^{*}+(1+\gamma m_{R})c_{R}>0.  
\end{equation*}
For (ii), we have $1+\gamma m_{R}>0$ when $-1\leq\gamma$, and thus it is clear from (\ref{HR}) that $H(R)>0$ for all $R>0$. The fact that $\eta_{Q,R}$ is a positive measure 
follows from (\ref{sgn-meas}) since $-\gamma\geq0$ and $1+\gamma m_{R}>0$.
\end{proof}
\begin{theorem}
Let $\gamma_{-}<0$ be the unique root of $H(1)=0$, that is
\begin{equation}\label{gamma0}
\gamma_{-}=\frac{c_{1}}{\Lambda_{1}^{*}-m_{1}c_{1}}<-1.
\end{equation}
(i) For $\gamma\leq\gamma_{-}$, let $R_{\gamma}$ be such that $H(R_{\gamma})=0$. Then $R_{\gamma}\leq1$ and
\begin{equation}\label{neg-small}
S_{Q,\bar\B}=B_{R_{\gamma}}\quad\text{ and }\quad\om_{Q,\bar\B}=\eta_{Q,{R_{\gamma}}}.
\end{equation}
The density of the equilibrium measure $\om_{Q,\bar\B}$ vanishes on the boundary $S_{R_{\gamma}}$ of its support.
\\[5pt]
(ii) For $\gamma_{-}<\gamma<0$, 
\begin{equation}\label{neg-large}
S_{Q,\bar\B}=\bar\B\quad\text{ and }\quad\om_{Q,\bar\B}=\eta_{Q,1}.
\end{equation}
The density of the equilibrium measure $\om_{Q,\bar\B}$ behaves like $M/(1-|x|^{2})^{\alpha/2}$ as $|x|\to1^{-}$ near the boundary $\mathbb{S}$ of the unit ball $\bar\B$, where $M$ is a positive constant.
\end{theorem}
\begin{proof}
The inequality in (\ref{gamma0}) follows from item (ii) of the previous lemma. For (i), $R_{\gamma}\leq1$ because it follows from (\ref{sol}) that $R_{\gamma}$ is an increasing function of $\gamma$ (because the left-hand side of (\ref{sol}) is an increasing function of $z$). Then, (\ref{neg-small}) follows from (ii) of Lemma \ref{inclusion} and (i) of Lemma \ref{lem-neg}. The vanishing of the density was proved in \cite[Theorem 4.1, item (ii)]{DOSW}.
We now prove (ii). For $\gamma_{-}<\gamma<-1$, $H(R)=0$ has a root $R_{\gamma}>1$ and (i) of the previous lemma shows that $\eta_{Q,1}$ is a positive measure, whence (\ref{neg-large}). The behavior of the density follows from the fact that $H(1)>0$. For $-1\leq\gamma<0$, the result is a consequence of (ii) of Lemma \ref{lem-neg}.
\end{proof}
\begin{figure}[htb]
\centering
  \includegraphics[scale=0.6]{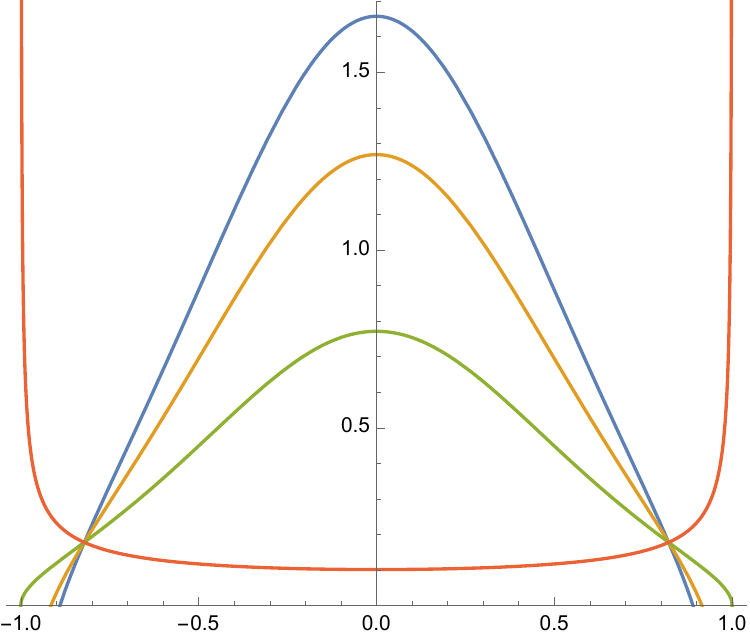}\hspace{.6cm}
\caption{Radial densities of the weighted Riesz equilibrium measures for the ball in $\R^{3}$ and for negative charges $\gamma\in\{-20,-15,\gamma_{-}=-8.612,0\}$ (resp.\ in blue, orange, green, red) when $d=3$ and $s=1$. The charge $\gamma$ is located at height $y_{4}=1$ above the ball.}
\end{figure}
\subsection{The case of a repulsive charge $\gam>0$}
\label{Pos}
%
%
%
%
%
%
%
In this section, we consider the case of an external field induced by a positive charge $\gamma>0$ at $y=(0,y_{d+1})$ above the $d$-dimensional ball $\bar\B$. We start with a lemma (that we will use only when $d=1$, see Section \ref{shell}).
\begin{lemma}\label{lem-quo-incr}
The function $(\eta'_{Q,1}/\om_{1}')(|x|)$ is a strictly increasing function of $|x|$.
\end{lemma}
\begin{proof}
In view of (\ref{sgn-meas}), (\ref{equilball}), (\ref{balayageBR}), and the fact that $\gamma>0$, the derivative of $(\eta'_{Q,1}/\om_{1}')(|x|)$ has opposite sign to that of 
\begin{align*}
\frac{(1 - |x|^2)^{\alpha/2}}{(|x|^2 + y_{d+1}^2)^{d-s/2}}+
 \pi^{-1}\sin(\alpha\frac{\pi}{2})
 \int_0^{+\infty}\frac{v^{\alpha/2}\,dv}{(v+1+y^2_{d+1})^{d-s/2} (v+1-|x|^2)}.
\end{align*}
We show this is a decreasing function of $|x|$, or equivalently, that
$$
\frac{u^{\alpha/2}}{(1+y_{d+1}^2-u)^{d-s/2}}+
 \pi^{-1}\sin(\alpha\frac{\pi}{2})
 \int_0^{+\infty}\frac{v^{\alpha/2}\,dv}{(v+1+y^2_{d+1})^{d-s/2} (v+u)}
$$
is an increasing function of $u=1-|x|^{2}$. Taking the derivating, and omitting the derivative of the denominator in the first term, which is already positive, we obtain
$$
\frac{(\alpha/2)u^{\alpha/2-1}}{(1+y_{d+1}^2-u)^{d-s/2}}-
 \pi^{-1}\sin(\alpha\frac{\pi}{2})
 \int_0^{+\infty}\frac{v^{\alpha/2}\,dv}{(v+1+y^2_{d+1})^{d-s/2} (v+u)^{2}}.
$$
Since $1+y_{d+1}^{2}-u\leq1+y_{d+1}^{2}+v$, it is sufficient to show that
$$
(\alpha/2)u^{\alpha/2-1}-
 \pi^{-1}\sin(\alpha\frac{\pi}{2})
 \int_0^{+\infty}\frac{v^{\alpha/2}\,dv}{(v+u)^{2}}\geq0.
$$
But the above term equals
$$
(\alpha/2)u^{\alpha/2-1}-\pi^{-1}\sin(\alpha\frac{\pi}{2})
u^{\alpha/2-1}B(1+\frac\alpha2,1-\frac\alpha2)=
u^{\alpha/2-1}\left(\frac\alpha2-\frac{\Gamma(1+\frac\alpha2)\Gamma(1-\frac\alpha2)}
{\Gamma(\frac\alpha2)\Gamma(1-\frac\alpha2)}\right)=0.
$$
\end{proof}
\begin{theorem}\label{Thm-pos-small}
Let $\gamma_{+}$ be the unique root of $\eta_{Q,1}'(0)=0$, that is
$$
\gamma_{+}=\frac{c_{1}}{Bal'(\delta_{y},\bar\B)(0)-m_{1}c_{1}}>0.
$$
For $0\leq\gamma\leq\gamma_{+}$,
\begin{equation}\label{pos-small}
S_{Q,\bar\B}=\bar\B\quad\text{ and }\quad\om_{Q,\bar\B}=\eta_{Q,1}.
\end{equation}
The density of the equilibrium measure $\om_{Q,\bar\B}$ behaves like $M/(1-|x|^{2})^{\alpha/2}$ as $|x|\to1^{-}$ near the boundary $\mathbb{S}$ of its support, where $M$ is some positive constant.
\end{theorem}
\begin{proof}
When $\gamma\leq0$ we know that $\eta_{Q,1}$ is a positive measure with a positive density, implying $\eta_{Q,1}'(0)>0$. Hence we must have $\gamma_{+}>0$ (note that $Bal'(\delta_{y},B)(0)-m_{1}c_{1}\neq0$ since otherwise $\eta_{Q,1}'(0)$ would be independent of $\gamma$, which is not the case). In particular, $Bal'(\delta_{y},B)(0)-m_{1}c_{1}>0$ which shows, together with  (\ref{sgn-meas}) that $\eta_{Q,1}'(0)$ is a decreasing function of $\gamma$. Consequently, for $0\leq\gamma\leq\gamma_{+}$, $\eta_{Q,1}'(0)\geq0$ which implies, in view of Lemma \ref{lem-quo-incr}, that $\eta_{Q,1}$ is a positive measure, from which (\ref{pos-small}) follows. The behavior of the density near the boundary also follows from Lemma \ref{lem-quo-incr} since the quotient $(\eta_{Q,1}'/\om_{1}')(|x|)$ is positive for $x\neq0$, in particular as $|x|\to1^{-}$.
\end{proof}
\begin{figure}[htb]
\centering
  \includegraphics[scale=0.7]{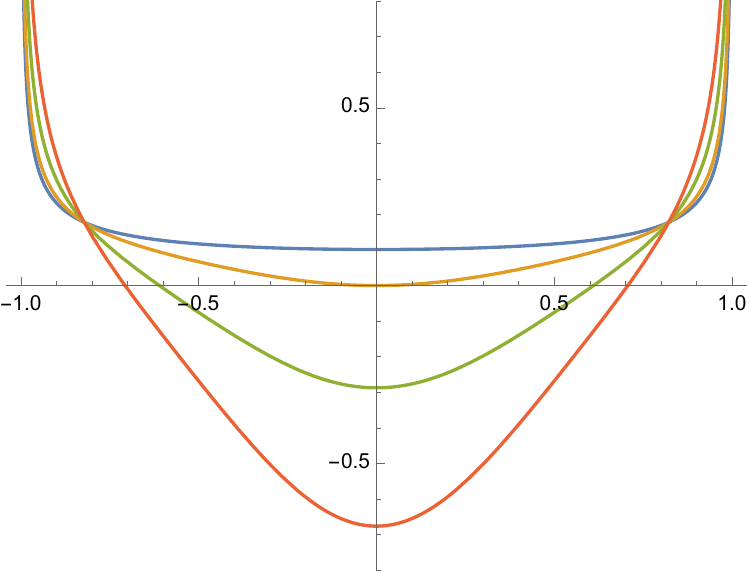}\hspace{.6cm}
\caption{Radial densities of the Riesz signed equilibrium measures for the ball in $\R^{3}$ and for positive charges $\gamma\in\{0,\gamma_{+}=1.302,5,10\}$ (resp.\ in blue, orange, green, red) when $d=3$ and $s=1$. The charge $\gamma$ is located at height $y_{4}=1$ above the ball.}
\end{figure}
When $\gamma_{+}<\gamma$, one has $\eta_{Q,1}'(0)<0$ and thus $S_{Q,\bar\B}$ does not contain the origin. Also, the outer boundary of $S_{Q,\bar\B}$ is the unit sphere $\mathbb{S}$. Indeed, expanding radially a measure contained in the interior of $\mathbb{S}$ can only make its weighted energy decrease (since $\gamma>0$).

We make the following conjecture :
\begin{Conjecture}[Shell Conjecture] \label{Conjec}
Assume $d\geq1$ and $\max(0,d-2)<s<d$. For a repulsive charge $\gamma>0$,
the support $S_{Q,\bar\B}$ is either the entire ball $\bar\B$ or a shell with inner boundary $S_{r}$, $0<r<1$, and outer boundary $\mathbb{S}$. 
\end{Conjecture}
Section \ref{shell} will be devoted to a proof of the conjecture in the case of a segment (one-dimensional case $d=1$).
\section{The Coulomb case $s=d-2$ for a ball in~$\R^{d}$, $d\geq2$}\label{Coulomb}
For the Coulomb case $s=d-2$ with $d\geq 2$ (with logarithmic interaction when $d=2$), 
 the external field $Q(x) = Q(|x|) = Q(r)$ is given by
\begin{equation}\label{QRiesz}
Q(r) = \,\frac{\gamma}{(r^2+h^2)^{d/2-1}},\quad\text{if }d\geq3,    \qquad
Q(r) = \frac\gamma2\log\frac{1}{(r^2+h^2)},\quad\text{if }d=2.
\end{equation}
We set
\begin{equation}\label{R0Riesz}
R_0 = \,\frac{h}{\sqrt{|\gamma|^{2/d} -1}}\,,\quad\gamma < -1,
\end{equation}
\begin{equation}\label{RRiesz}
\tilde\gamma=-(1+h^2)^{d/2},\qquad R = \begin{cases}
R_0,& \text{if }\gamma < \tilde\gamma,\\
1,& \text{if }\tilde\gamma \leq \gamma < 0\,,
\end{cases}
\end{equation}
and 
\begin{equation}\label{tauRiesz}
d\tau (x) = - d \gamma h^2\,\frac{r^{d-1}}{(r^2+h^2)^{1+d/2}}\,d r d\sigma_{d-1},\qquad
0\leq r \leq R\,,    
\end{equation}
where $\sigma_{d-1}$ denotes the normalized surface area measure of the $(d-1)$-dimensional unit sphere $\mathbb{S} \subset \R^d$. 

We have the following result. 
\begin{theorem}\label{thm:equildisk2}
Let $d\geq 2$ and $s=d-2$. For the equilibrium problem on the unit ball $\B \subset \R^d$ in the external field \eqref{QRiesz}, we have the following 
\begin{itemize}
    \item[(i)] If $\gamma \leq\tilde\gamma$, then $R\leq 1$ and
    $$d\mu_{Q,\B}(x) = d\tau (x),\quad x\in B_{R}.$$
    \item[(ii)] If $ \tilde\gamma< \gamma < 0$, then $R=1$ and
    $$d\mu_{Q,\B}(x) = d\tau (x) + \left(1 -\gamma/\tilde\gamma \right)d\sigma_{d-1}(x),\quad x\in \B.$$
    \item[(iii)] If $\gamma \geq 0$, then
    $$d\mu_{Q,\B}(x) = d\sigma_{d-1} (x),\quad x\in \mathbb{S}.$$
\end{itemize}
\end{theorem}
\begin{remark}\label{rem:mix}
    Observe that in case (ii), the equilibrium measure is made of two components, a continuous part supported on $\B$ and a singular part supported on its boundary $\mathbb{S}$. That type of mixture  appears similarly in \cite{CSW}, also for a radial external field.
\end{remark}
\begin{remark}\label{unweighted}
Note that the threshold value between cases (ii) and (iii), that is when $\gamma = 0$, corresponds to the Robin (unweighted) equilibrium problem in $\B$, with solution $\sigma_{d-1}$.
\end{remark}
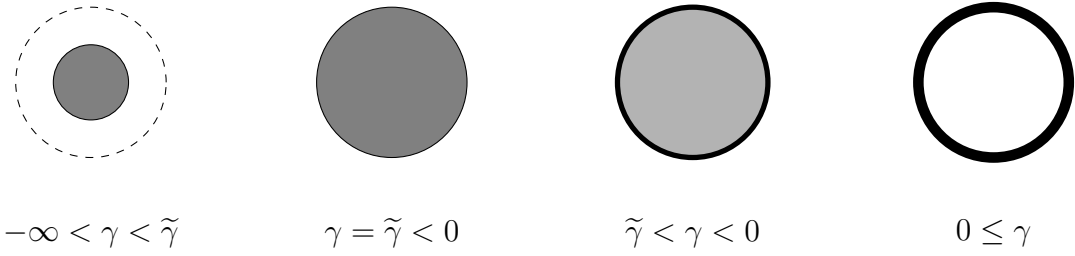
\begin{figure}
\begin{center}
\begin{tikzpicture}
\draw[fill=gray](-6,.0) circle (.5);
\draw[fill=none,dashed](-6,.0) circle (1);
\draw[fill=gray](-2,.0) circle (1);
\draw[fill=gray!60,line width=2pt](2,.0) circle (1);
\draw[fill=none,line width=4pt](6,.0) circle (1);
\draw (-6,-2) node {$-\infty<\gamma<\tilde\gamma$};
\draw (-2,-2) node {$\gamma=\tilde\gamma<0$};
\draw (2,-2) node {$\tilde\gamma<\gamma<0$};
\draw (6,-2) node {$0\leq\gamma$};
\end{tikzpicture}
\caption{Evolution of the equilibrium measure $\mu_{Q,\B}$ as described in Theorem \ref{thm:equildisk2}. For $\gamma<\tilde\gamma$, a large attractive charge, $\mu_{Q,\B}$ is supported on a ball of radius $<1$, which grows with $\gamma$. When $\gamma=\tilde\gamma$, the radius becomes equal to 1. For $\tilde\gamma<\gamma<0$, $\mu_{Q,\B}$ is a combination of a measure on $\B$ and a multiple of the uniform measure on the sphere $\mathbb{S}$. When $\gamma=0$, all of the volume measure has moved to the sphere, and $\mu_{Q,\B}$ is just the normalized uniform measure on $\mathbb{S}$. Then, for a positive charge $\gamma>0$, $\mu_{Q,\B}$ remains unchanged.}
\end{center}
\end{figure}
%
We first state a preliminary result, which holds true in the case of Coulomb interaction, see \cite[Theorem IV.6.1]{ST} for $d=2$, and \cite[Proposition 2.13]{abey} for $d\geq3$.
\begin{lemma}\label{lem:abeyrad}
Let $\Sigma = \R^d$ and let $Q$ be a radial external field, $Q(x) =  Q(|x|) = Q(r)$ which has an absolutely continuous derivative and satisfies 
\begin{equation}\label{cond-adm-2}
\lim_{r\to\infty}Q(r)=+\infty\quad\text{if }d\geq3,\qquad
\lim_{r\to\infty}Q(r)-\log r=+\infty\quad\text{if }d=2.
\end{equation}
Assume that one of the following conditions holds:
\begin{itemize}
    \item [(a)] $r^{d-1}Q'(r)$ is increasing on $(0,\infty)$,
    \item [(b)] $Q(r)$ is convex on $(0,\infty)$.
\end{itemize}
Let $r_0 \geq 0$ be the smallest number such that $Q'(r)>0$ for $r>r_0$, and $R_0 > 0$ the smallest solution of the equation 
$$R_0^{d-1}Q'(R_0) = 
\begin{cases}
d-2,&\text{if }d\geq3,
\\
1,&\text{if }d=2.
\end{cases}
$$ 
Then, the support of the equilibrium measure $\mu_{Q,\R^d}$ is given by
$$S_{Q,\R^d} = \,\{x\in \R^d\,: \,r_0 \leq |x| \leq R_0\}$$
and its density by
$$d\mu_{Q,\R^d}(x) = \min(1,\frac{1}{d-2})(r^{d-1} Q'(r))'\,dr\,d\sigma_{d-1},$$
where $\sigma_{d-1}$ stands for the normalized surface measure of the unit sphere $\mathbb{S} \subset \R^d$.   
\end{lemma}
\begin{remark} The conditions of admissibility (\ref{cond-adm-2}) could be substituted, respectively,
with the weaker conditions,
$$
\lim_{r\rightarrow \infty}\,r^{d-2}\,(Q(r) - Q(\infty)) \leq -1, \quad\text{if }d\geq3,\qquad
\liminf_{r\to\infty}Q(r)-\log r>-\infty\quad\text{if }d=2,
$$
see \cite[Theorem 2.4]{DOSW} and \cite[Theorem 3.4]{BLW}.
\end{remark}
\begin{proof}[Proof of Theorem \ref{thm:equildisk2}]
We first consider the logarithmic case, that is $d=2$ and $s=0$. 
Note that the condition of admissibility (\ref{cond-adm-2}) in the previous lemma is satisfied only when 
$\gamma < -1$.
 For $\gamma = -1$, the equilibrium measure in the whole of $\C$ still exists (weakly admissible case) but its support is then unbounded. 
We have that
$$r Q'(r) = -\gamma \,\frac{r^2}{r^2+h^2}\,,\qquad(r Q'(r))' = -\gamma \,\frac{2h^2r}{(r^2+h^2)^2}\,,$$
so that the function $r Q'(r)$ is increasing if $\gamma <0$ (an attractive charge). Hence condition (a) in Lemma \ref{lem:abeyrad} holds and an easy computation yields
$$r_0 = 0\,,\; R_0 = \,\frac{h}{\sqrt{-(1+\gamma})}\,=\,\frac{h}{\sqrt{|\gamma|-1}}.$$
Thus, for $\gamma<-1$, the support $S_{\mu_Q,\C}$ is a ball $B_{R_0}$ centered at the origin and with radius $R_0$. Therefore, for $\gamma \leq -1-h^2 < -1$ we have $R_0 < 1$ and $R=R_0$, so that $\mu_{Q,\C} = \tau$. Therefore, for the equilibrium problem in $\bar\D$, we have 
$S_{Q,\overline{\D}} = B_{R_0}$, and $\mu_{Q,\overline{\D}}=\tau$. Hence, part (i) is proved.

Next, consider case (ii), where $-1-h^2 < \gamma < 0$. 
The measure $\tau$ in \eqref{tauRiesz}, now supported on $\bar\D$, appears to be still useful in finding the density of $\mu_{Q,\overline{\D}}$. Indeed, 
$$\tau(\bar\D) = \frac{-\gamma h^2}{\pi}\,\int_{-\pi}^{\pi}\,d\theta \; \int_0^1\,\frac{r}{(r^2+h^2)^2\,dr}\,= \,-\frac{\gamma}{h^2+1}<1$$
and thus it is not a probability measure.
However, 
$$
\nu = \tau + (1+ \frac{\gamma}{1+h^2})\,\sigma
$$ is a probability measure, with potential
$$U^{\nu}(z) = \frac{1}{2\pi}\int_{\D}\,\log \frac{1}{|z-y|}\,d\nu(y).$$
Using polar coordinates and the well known identity (see e.g. \cite[(0.5.5)]{ST})
\begin{equation}\label{poissonk}
\frac{1}{2\pi}\,\int_{-\pi}^{\pi}\,\log \frac{1}{|z-re^{i\theta}|}d\theta=\,\begin{cases}
    - \log r,\; & |z|\leq r, \\
    - \log |z|,\; & |z|>r,
\end{cases}    
\end{equation}
we get that
$$U^{\tau}(z) = -\log |z|\,\int_0^{|z|}\,(r Q'(r))' dr\, - \int_{|z|}^1\,(r Q'(r))'\,\log r\, dr\,,$$
and thus, after straightforward computations, it holds
$U^{\nu}(z) = U^{\tau}(z)=Q(1) - Q(z)$, i.e.
$$U^{\nu}(z) + Q(z) = Q(1) = \,-\gamma \log (1+h^2),$$
which is constant. Frostman's identity thus holds true, which proves part (ii).

Finally, as for the case (iii), where $\gamma > 0$, it is easy to check that
$$U^{\sigma}(z) + Q(z) = \begin{cases}
    Q(1),\; & |z| = 1, \\ 
    Q(z),\; & |z| < 1.
\end{cases}$$
Since, for $\gamma > 0$, $Q(r)$ is a decreasing function of $r$, it proves that Frostman's inequalities hold true.

The proof when $d\geq3$ is completely similar to the case $d=2$. 
We only note that,
in part (ii), the total mass of $\tau$ in \eqref{tauRiesz} is given by
$$\tau (B_{1}) = \frac{1}{d-2}\,\int_0^1\,(r^{d-1}\,Q'(r))'dr =\,\frac{Q'(1)}{d-2}=
\gamma/\tilde\gamma, $$
and \eqref{poissonk} can be replaced with the identity \cite[p.21]{abey},
$$\int_{S^{d-1}}\,\frac{1}{|r {y} - x|^{d-2}}\,d\sigma_{d-1}({y})= \begin{cases}
    {r^{2-d}},\; & |x|\leq r, \\[5pt]
    {|x|^{2-d}},\; & |x|>r,
\end{cases}$$
to compute the potential of $\tau$.
\end{proof}
\section{The logarithmic case for the segment $I=[-1,1]$}\label{log-seg}
The signed equilibrium measure on $I_{r}=[-r,r]$, $r>0$, is
$$
\eta_{Q,{r}}=-\gamma Bal(\delta_{y},I_{r})+(1+\gamma)\om_{r},
$$
with density 
\begin{equation}\label{bal-Ir}
\eta_{Q,{r}}'(x)=\left(1+\gamma
-\frac{\gamma y_{2} \sqrt{y_{2}^{2}+r^{2}}}{\left(x^{2}+y_{2}^{2}\right)}\right)
\frac{1}{\pi\sqrt{r^{2}-x^{2}}},\qquad x\in I_{r},
\end{equation}
where we have used \cite[p.958]{BDT} for an explicit expression of the density of the balayage $Bal(\delta_{y},I_{r})$. For $r=1$, the first factor vanishes when
$$
x^{2}=\frac{\gamma y_{2}\sqrt{y_{2}^{2}+1}}{1+\gamma}-y_{2}^{2}.
$$
Hence, the density $\eta_{Q,I}'(x)$ vanishes in $[-1,1]$ if and only if
\begin{equation}\label{ineq-I}
0\leq\frac{\gamma y_{2}\sqrt{y_{2}^{2}+1}}{1+\gamma}-y_{2}^{2}\leq1.
\end{equation}
\subsection{Case of an attractive charge $\gam\leq0$}
1) If $\gamma<-1$, (\ref{ineq-I}) becomes
$$
0\geq{\gamma y_{2}\sqrt{y_{2}^{2}+1}}-(1+\gamma)y_{2}^{2}\geq1+\gamma
$$
or, equivalently
\begin{multline*}
\begin{cases}
y_{2}^{2}\geq\gamma(y_{2}\sqrt{y_{2}^{2}+1}-y_{2}^{2}),
\\[10pt]
\gamma(y_{2}\sqrt{y_{2}^{2}+1}-y_{2}^{2}-1)\geq1+y_{2}^{2}
\end{cases}
\\\quad\iff\quad
\gamma\leq\frac{y_{2}}{\sqrt{y_{2}^{2}+1}-y_{2}}\quad\text{and}\quad
\gamma\leq\frac{\sqrt{1+y_{2}^{2}}}{y_{2}-\sqrt{y_{2}^{2}+1}}.
\end{multline*}
The first inequality is always satisfied, and the last quotient is less than $-1$. Hence, if
\begin{equation}\label{eq-gam-moins}
\gamma<\gamma_{-}:=\frac{\sqrt{1+y_{2}^{2}}}{y_{2}-\sqrt{y_{2}^{2}+1}}<-1,
\end{equation}
the density $\eta_{Q,I}'$ is positive on $[-r^{*},r^{*}]$, and negative on $I\setminus[-r^{*},r^{*}]$, where $0<r^{*}\leq1$ satisfies
\begin{equation}\label{eq-r*}
(r^{*})^{2}=\frac{\gamma y_{2}\sqrt{y_{2}^{2}+1}}{1+\gamma}-y_{2}^{2}.
\end{equation}
Moreover, $\eta_{Q,r^{*}}$ is a positive measure because the first factor in (\ref{bal-Ir}) is decreasing with $r$. Hence, $\om_{Q,I}=\eta_{Q,r^{*}}$.
\\[5pt]
2) If $\gamma_{-}\leq\gamma<-1$, $\eta_{Q,I}$ is a positive measure and $\om_{Q,I}=\eta_{Q,I}$.
\\[5pt]
3) If $\gamma=-1$, it follows from (\ref{bal-Ir}) that
$$
\eta'_{Q,I}(x)=\frac{y_{2}\sqrt{1+y_{2}^{2}}}{\pi\sqrt{1-x^{2}}(x^{2}+y_{2}^{2})}>0,\qquad x\in(-1,1),
$$
which implies that $\om_{Q,I}=\eta_{Q,I}$.
\\[5pt]
4) If $-1<\gamma<0$, (\ref{ineq-I}) becomes
\begin{multline}\label{2-ineq}
0\leq{\gamma y_{2}\sqrt{y_{2}^{2}+1}}-(1+\gamma)y_{2}^{2}\leq1+\gamma
\quad\iff\quad
\\
\gamma\geq\frac{y_{2}}{\sqrt{y_{2}^{2}+1}-y_{2}}=y_{2}(\sqrt{y_{2}^{2}+1}+y_{2})\quad\text{and}\quad
\gamma\geq\frac{\sqrt{1+y_{2}^{2}}}{y_{2}-\sqrt{y_{2}^{2}+1}}=-(1+y_{2}^{2}+y_{2}\sqrt{y_{2}^{2}+1}).
\end{multline}
The first inequality is never satisfied, which implies that $\eta_{Q,I}$ is a positive measure and thus
$\om_{Q,I}=\eta_{Q,I}$ again. 

We can summarize the findings above in the following Theorem.
\begin{theorem}
Let $\gamma$ be a negative charge.\\
1) If $\gamma<\gamma_{-}$ (given by (\ref{eq-gam-moins})) then $\om_{Q,I}$ has support $[-r^{*},r^{*}]$ with $r^{*}$ given by (\ref{eq-r*}), and the density given by (\ref{bal-Ir}) with $r=r^{*}$.
\\
2) If $\gamma_{-}\leq\gamma<0$, then $\om_{Q,I}$ has support $[-1,1]$ and the density still given by (\ref{bal-Ir}) with $r=1$.
\end{theorem}
\subsection{Case of a repulsive charge $\gam>0$}
The two last inequalities in (\ref{2-ineq}) are still valid, with the second one being always satisfied.
\\[5pt]
1) If $0<\gamma\leq\gamma_{+}$ with
\begin{equation}\label{def-gam-plus}
\gamma_{+}:={y_{2}}(\sqrt{y_{2}^{2}+1}+y_{2}),
\end{equation}
then $\eta_{Q,I}$ is still a positive measure and
$\om_{Q,I}=\eta_{Q,I}$. Note that, when $\gamma=\gamma_{+}$, the density of $\eta_{Q,I}$ vanishes at the origin.
\\[5pt]
2) If $\gamma>\gamma_{+}$, the density $\eta_{Q,I}'$ is negative on $[-r^{*},r^{*}]$, and positive on $K_{r^{*}}=I\setminus[-r^{*},r^{*}]$, where $0<r^{*}\leq1$ satisfies (\ref{eq-r*}). This implies that the support of $\om_{Q,I}$ is a subset of $K_{r^{*}}$.

To go further in computing the support and the density of $\om_{Q,I}$, we will make use of a different method, not based on balayage, but on complex analysis and the theory of singular integral equations, which can be applied in the case of the logarithmic kernel, see \cite{DKM, MFRa, Rakh, MOR} for some of the main related references. This method was already used in \cite{OS, OSW} to study equilibrium problems in the presence of external fields with rational derivatives.
The following proposition can be derived from \cite[Proposition 2.51]{DKM}. \begin{proposition}\label{lem:MOR}
Let 
$$
\Phi(z)=-\frac\gamma2\log(z^{2}+y_{2}^{2}),
$$
be the analytic continuation of $Q(x)$ to $\C$, and let $\widehat \om_{Q,I}$ be the Cauchy transform of the equilibrium measure $\om_{Q,I}$,
$$\widehat {\om}_{Q,I} (z) = \int\frac{d\om_{Q,I}(x)}{x-z}.$$
Then
\begin{equation}\label{algeb1}
((\widehat {\om}_{Q,I})_{\pm} (x) + \Phi'(x))^2 = R(x),
 \qquad x\in(-1,1),
\end{equation}
where $(\widehat {\om}_{Q,I})_{\pm}$ stands for the limit of the Cauchy transform when $z$ tends to the real interval $(-1,1)$ from above and below, $R(x)$ is the rational function,
\begin{equation}\label{def-R}
R(x)=(1+\gamma)^{2}\frac{x^2{(\tilde r^{2}-x^{2})}}{(x^2+y_{2}^2)^2(1-x^2)},
\end{equation}
and $[-1,-\tilde r]\cup[\tilde r,1]$ is the support of $\om_{Q,I}$.
Moreover, the density of the (absolutely continuous) equilibrium measure $\om_{Q,I}$ is given by
\begin{equation}\label{exp-om-Q}
\om'_{Q,I} (x) = \frac{1}{\pi}\,\sqrt{|R(x)|},\qquad x\in \supp\om_{Q,I}.
\end{equation}
\end{proposition}
Applying the proposition for the case $\gamma>\gamma_{+}$, we get a complete description of the equilibrium measure for a repulsive charge.
\begin{theorem}
Let $\gamma$ be a positive charge.\\
1) If $\gamma\leq\gamma_{+}$ (given by (\ref{def-gam-plus})) then $\om_{Q,I}$ has support $[-1,1]$ and the density given by (\ref{bal-Ir}) with $r=1$.
\\
2) If $\gamma>\gamma_{+}$ then $\om_{Q,I}$ has support $K_{\tilde r}$ with
the inner endpoint 
    \begin{equation}\label{asol}
    \tilde r = \frac{\sqrt{\gamma^2-(2\gamma+1)y_{2}^2}}{1+\gamma},
    \end{equation}
    and the density given by
    \begin{equation}\label{denslogrep}
    \om'_{Q,I}(x) = \,\frac{1+\gamma}{\pi}\frac{|x| \sqrt{x^2-\tilde r^2}}{(x^2+y_{2}^2) \sqrt{1-x^2}},
    \qquad \tilde r<|x|<1.    
    \end{equation}
\end{theorem}
\begin{remarks}
1) When $\gamma=\gamma_{+}$ we have, according to the language in random matrix theory, 
a {\it phase transition}, since the support of the equilibrium measure passes from the so-called one-cut regime (a single interval) to a two-cut regime (two disjoint intervals). The density of the equilibrium measure evolves continuously through this transition; however, the equilibrium energy presents a discontinuity, known in the literature as singularity of type I (birth of cut, see e.g.\ \cite[Section 4.1]{MOR}.
\\
2) When $\gamma>\gamma_{+}$, note the behavior of the density at the endpoints in \eqref{denslogrep}. It vanishes as a square root as we approach ``soft'' endpoints $\pm\tilde r$. 
On the contrary, the density tends to infinity as the reciprocal of a square root as we approach the ``hard'' endpoints $\pm 1$ (i.e. the endpoints of the conductor).
\end{remarks}
\begin{proof}
1) was derived at the beginning of the section.\\
2) The expression (\ref{denslogrep}) for the density $\om_{Q}'$ just follows from (\ref{exp-om-Q}) and (\ref{def-R}). To find the value (\ref{asol}) of $\tilde r$, one can use the fact that the total mass of $\om_{Q}$ is 1, which by symmetry, gives the condition
\begin{equation}\label{eq-mass1}
\frac{1+\gamma}{\pi}\int_{\tilde r}^{1}\frac{x\sqrt{x^2-\tilde r^2}}{(x^2+y_{2}^2) \sqrt{1-x^2}}dx=\frac12.
\end{equation}
The above integral can be rewritten, with $s=x^{2}$, and $v^{2}=(s-\tilde r^{2})/(1-s)$,
\begin{align*}
\frac12\int_{\tilde r^{2}}^{1}\frac{\sqrt{s-\tilde r^2}}{(s+y_{2}^2) \sqrt{1-s}}ds & =
(1-\tilde r^{2})\int_{0}^{\infty}\frac{v^{2}dv}{(v^{2}+1)((1+y_{2}^{2})v^{2}+(\tilde r^{2}+y_{2}^{2})}
\\[5pt]
& =\frac{1-\tilde r^{2}}{1+y_{2}^{2}}\int_{0}^{\infty}\frac{v^{2}dv}{(v^{2}+1)(v^{2}+c)},\qquad\text{where }
c=\frac{\tilde r^{2}+y_{2}^{2}}{1+y_{2}^{2}}.
\end{align*}
The last integral can be easily evaluated and we get for the above expression,
$$
\frac{1-\tilde r^{2}}{(1+y_{2}^{2})(c-1)}\left[-\arctan v+\sqrt{c}\arctan(v/\sqrt{c})\right]_{0}^{\infty}
=\frac\pi2\frac{1-\tilde r^{2}}{(1+y_{2}^{2})(\sqrt{c}+1)}
=\frac\pi2\frac{1-c}{\sqrt{c}+1}.
$$
Equation (\ref{eq-mass1}) becomes
$$
{(1-\sqrt{c})(1+\gamma)}=1
$$
which readily gives (\ref{asol}).
\end{proof}
\begin{figure}[htb]
\centering
  \includegraphics[scale=0.6]{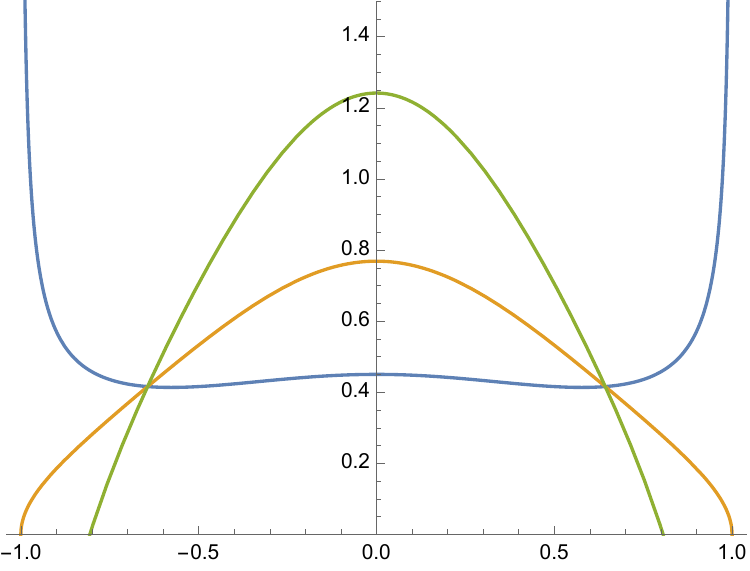}\hspace{.6cm}
  \includegraphics[scale=0.6]{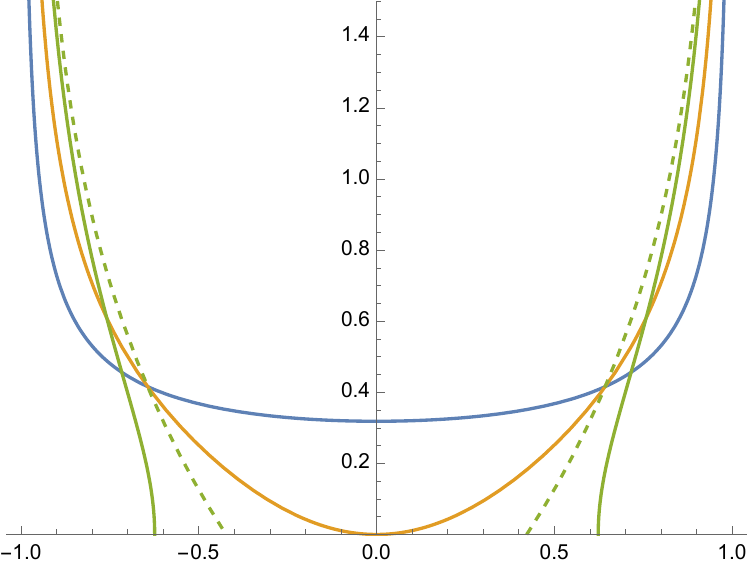}
\caption{Radial densities of the logarithmic equilibrium measures for negative charges $\gamma\in\{-1,\gamma_{-}=\sqrt{2}/(1-\sqrt{2}),-7\}$ (in blue, orange, green) on the left, and for positive charges $\gamma\in\{0,\gamma_{+}=\sqrt{2}+1,5\}$ (in blue, orange, green) on the right. The charge is located at height $y_{2}=1$ above the segment. On the right, the dashed curve is the density of the positive component of the signed equilibrium measure for $\gamma=5$.}
\end{figure}
\section{Proof of the shell conjecture when $d=1$}\label{shell}
In this section, we consider the segment $I=[-1,1]\subset\R$, assuming $0<\alpha<1$ (or equivalently $0<s<1$). We study the case of a charge $\gamma>0$ whose strength is larger than the critical value from Theorem \ref{Thm-pos-small}, that is $\gamma_{+}<\gamma$. We show, in a nontrivial way, that $S_{Q,I}$ is the union of two symmetric segment, namely, of the form
$$
S_{Q,I}=K_{r}:=[-1,-r]\cup[r,1],
$$
for some $0<r<1$, that is we solve Conjecture \ref{Conjec}, see Corollary \ref{supp-shell} for a statement. To achieve our goal we will follow the scheme of proof used in \cite{KD}, based on a sequence of iterative balayages, starting with the signed equilibrium measure $\eta_{Q,I}$, and sweeping out recursively the negative part of the measures onto their positive part. We show that the sequence tends to a positive limit measure, which, up to mass loss, is the equilibrium measure $\om_{Q,I}$.

Note that all the results obtained in this section hold true for the logarithmic kernel, that is in the case $s=0$ or, equivalently, $\alpha=1$.

We first need several lemmas.
\begin{lemma}[{\cite[Lemma 2.2]{SKM}}]\label{Samko}
Let $f$ be an absolutely continuous function on a segment $[a,b]$, and let $0<s<1$. Then the derivatives
$$
\frac{d}{dx}\int_{a}^{x}\frac{f(t)dt}{(x-t)^{s}},\qquad\frac{d}{dx}\int_{x}^{b}\frac{f(t)dt}{(t-x)^{s}},
$$
exist almost everywhere, and they are respectively equal to
$$
\frac{f(a)}{(x-a)^{s}}+\int_{a}^{x}\frac{f'(t)dt}{(x-t)^{s}},\qquad
-\frac{f(b)}{(b-x)^{s}}+\int_{x}^{b}\frac{f'(t)dt}{(t-x)^{s}}.
$$
\end{lemma}
\begin{lemma}\label{2seg}
Let $K=[a,b]\cup[c,d]$ be the union of two segments in $\R$, with Riesz equilibrium measure $\om_{K}$. Then $\om_{K}'/\om_{[a,b]}'$ is decreasing on $[a,b]$ and 
$\om_{K}'/\om_{[c,d]}'$ is increasing on $[c,d]$.
\end{lemma}
\begin{proof}
The restriction $\om_{K}|_{[a,b]}$ can be seen as the equilibrium measure on $[a,b]$ in the external field created by $\om_{K}|_{[c,d]}$. Moreover, the corresponding signed equilibrium measure
\begin{equation}\label{sign-a-b}
-Bal(\om_{K}|_{[c,d]},[a,b])+(1+m)\om_{[a,b]},
\end{equation}
with $m$ the mass of $Bal(\om_{K}|_{[c,d]})$,
 is positive because $\supp\om_{K}=K$ (recall Theorem \ref{Wal}), in particular, $\supp\om_{K}|_{[a,b]}=[a,b]$. Thus $\om_{K}|_{[a,b]}$ equals (\ref{sign-a-b}).
The balayage of a point mass $\delta_{t}$ on $[a,b]$ satisfies
$$
(|x-b||x-a|)^{\alpha/2}\hat\delta_{t}(x)= \gamma_{s}\left({|t-b||t-a|}\right)^{\alpha/2}\frac{dx}{|x-t|},
\qquad\gamma_{s}=\pi^{-1}\sin(\alpha\pi/2),
$$
which increases on $[a,b]$ when $t$ is on the right of $[a,b]$, in particular when $t\in[c,d]$. This shows, together with the superposition principle, the first assumption. The second assertion is showed similarly (or follows by symmetry).
\end{proof}
\begin{remark}
We do not know if a version of this result holds in the case of a shell in $\R^{d}$, $d\geq2$, the main difference with the one-dimensional case being that a shell is a connected set. The fact that $K=[a,b]\cup[c,d]$ has two components was used in the above proof to make a link between the unweighted equilibrium measure on $K=[a,b]\cup[c,d]$ and a weighted equilibrium measure on $[a,b]$.
\end{remark}
\begin{lemma}\label{variat-bal}
Let $\om_{I}$ be the equilibrium measure of $[-1,1]$,
and let $\nu$ be a positive measure supported in $(-r,r)$. Then, the balayage measure $\hat\nu$ on $K_{r}$ is absolutely continuous with respect to the Lebesgue measure, and its density $\hat\nu'$ satisfies that
\begin{equation*}
(\hat\nu'/\om_{I}')(|x|)
\end{equation*}
is a decreasing function of $|x|$, $x\in K_{r}$. 
\end{lemma}
Note that the normalization by $\om_{I}'$ does not depend on $r$, in particular, it is independent from the set onto which the balayage is performed.
\begin{proof}
The proof is based on Lemma \ref{2seg}.
Again, it is sufficient to assume $\nu=\delta_{a}$, $a\in(-r,r)$, is a point mass. By a translation, we consider the balayage of $\delta_{0}$ on $K=[a,b]\cup[c,d]$ with $a<b<0<c<d$.
By Lemma \ref{Bal-eq}, with $K_{0}(t)=1/t$, the Kelvin transform of center $0$ and radius 1, and for $t\in[a,b]$, we have
$$
\frac{Bal(\delta_{0},K)'(t)}{\om'_{[a,d]}(t)}=C_{0,1}
\frac{\om_{K_{0}(K)}'(K_{0}(t))}{\om'_{[1/b,1/a]}(K_{a}(t))}
\frac{|t|^{s-2}\om'_{[1/b,1/a]}(K_{0}(t))}{\om'_{[a,d]}(t)}.
$$
In view of Lemma \ref{2seg}, the first quotient on the right hand-side is increasing as a function of $t$, and, up to some multiplicative constant, the second quotient equals
$$
|t|^{s-2}\left(\frac{(d-t)(t-a)}{(a^{-1}-t^{-1})(t^{-1}-b^{-1})}\right)^{\alpha/2}
=(ab)^{\alpha/2}|t|^{s-2+\alpha}\left(\frac{d-t}{b-t}\right)^{\alpha/2}
=\frac{(ab)^{\alpha/2}}{|t|}\left(\frac{d-t}{b-t}\right)^{\alpha/2},
$$
which is also increasing on $[a,b]$. The same reasoning can be performed on $[c,d]$.
\end{proof}
For a signed measure $\sigma$ with compact support on $\R$, and a nontrivial component $\sigma^{+}$, we define the map
\begin{equation}\label{def-J}
J(\sigma)=Bal(\sigma,\supp\sigma^{+})=\sigma^{+}-Bal(\sigma^{-},\supp\sigma^{+}).
\end{equation}
Note that, by 2) of Proposition \ref{bal-abs-cont}, if $\sigma$ is absolutely continuous with respect to the Lebesgue measure $dt$, then $J(\sigma)$ is also absolutely continuous. For ease of notation we simply denote by $\om_{R}$ the equilibrium measure $\om_{[-R,R]}$ of $[-R,R]$.
\begin{lemma}\label{lem-J}
Let $\sigma=\sigma^{+}-\sigma^{-}$ be a signed, absolutely continuous, measure 
on $K_{r}$, $0\leq r<1$, with $\sigma^{+}$ and $\sigma^{-}$ two nontrivial positive measures. Let $\sigma(x)=v(|x|)dx$ and assume 
$$(v/\om'_{1})(|x|)$$ 
is an increasing function of $|x|$, $x\in K_{r}$. Let $J(\sigma)=v^{*}(|x|)dx$. Then \\
1) $\exists\, r^{*}\in(r,1)$, $\supp(\sigma^{+})=K_{r^{*}},$ and $(v^{*}/\om'_{1})(|x|)$ is an 
increasing function of $|x|$ on $K_{r^{*}}$.\\
2) Assume moreover that the density of $\sigma$ is continuous. Denote by $c^{*}>0$ the smallest real number such that $(\sigma-c^{*}\om_{1})^{+}=0$. If $(\sigma-c\om_{1})^{+}$ is a nontrivial measure for any $c>0$, we set $c^{*}=+\infty$. Note that, for any $0\leq c<c^{*}$, $J(\sigma-c\om_{1})$ is well defined. The map 
$$F:c\in[0,c^{*})\mapsto m(J(\sigma-c\om_{1})),$$ 
is continuous, 
\begin{equation}\label{F-prop}
F(0)>m(\sigma),\qquad \limsup_{c\to c^{*}}F(c)\leq0.
\end{equation}
and there exists a $c_{\sigma}\in(0,c^{*})$ such that $F(c_{\sigma})=m(\sigma)$.
\end{lemma}
\begin{proof}
1) Since $\sigma$ has a nonzero positive part $\sigma^{+}$ and its density $v$, divided by the positive function $\om'_{1}$, is increasing with $|x|$, we must have $\supp\sigma^{+}=K_{r^{*}}$ with $r\leq r^{*}<1$ ($r^{*}\neq 1$ since $\sigma$ is absolutely continuous with respect to Lebesgue measure). Moreover, on $K_{r^{*}}$, we have $v^{*}=v-(\hat{\sigma_{-}})'$ where $(v/\om_{1}')(|x|)$ is an increasing function by assumption, and $((\hat{\sigma_{-}})'/\om'_{1})(|x|)$ is a decreasing function by Lemma \ref{variat-bal}. \\
2) The map $c\mapsto F(c)$ is continuous : the density of $\sigma-c\om_{1}$ vanishes at $t\in K_{r}$ when $(\sigma'/\om_{1}')(t)=(\sigma'/\om_{1}')(|t|)=c$, where $(\sigma'/\om_{1}')(|t|)$ is continuous and increasing. Thus, the previous equation has a unique positive root $r<t_{c}<1$ which depends continuously on $c$, with
$$
\supp(\sigma-c\om_{1})^{-}=K_{r}\setminus K_{t_{c}},\qquad\supp(\sigma-c\om_{1})^{+}=K_{t_{c}}.
$$
Thus,
$$
m((\sigma-c\om_{1})^{+})=\int_{K_{t_{c}}}d(\sigma-c\om_{1})
=\sigma(K_{t_{c}})-c\om_{1}(K_{t_{c}}),
$$
and
$$
m(Bal((\sigma-c\om_{1})^{-},K_{t_{c}}))=\capa(K_{t_{c}})\int U^{\sigma-c\om_{1}}d\om_{K_{t_{c}}},
$$
are continuous functions of $c$, where we use the fact that $\om_{K_{t_{c}}}$ is weak-* continuous with respect to $c$, and the potentials $U^{\sigma}$ and $U^{\om_{1}}$ are continuous. It follows that $c\mapsto F(c)$ is also continuous.

Next, we have $F(0)=m(J(\sigma))>m(\sigma)$ because $\sigma^{-}$ is a nontrivial measure. Moreover,
$$\limsup_{c\to c^{*}}F(c)\leq\limsup_{c\to c^{*}}m((\sigma-c\om_{1})^{+})=0.
$$
Finally, the last statement is just a consequence of (\ref{F-prop}) and the intermediate value theorem applied to the  continuous function $F$.
\end{proof}
Next we define a procedure that maps a signed measure $\sigma$ supported on $I$ into a positive 
measure, with support $K_{r^{*}}$, $r^{*}\in(r,1)$, with the property that its potential coincides with that of $\sigma$ on $K_{r^{*}}$, up to an additive constant.
\begin{proposition}\label{IBA}
Let $\sigma=\sigma^{+}-\sigma^{-}$ be a signed, absolutely continuous, radial measure 
on $K_{r}$, $0\leq r<1$, of positive total mass $m(\sigma)>0$, of finite energy, with $\sigma^{+}$ and $\sigma^{-}$, two nontrivial positive measures. Let $\sigma(x)=v(|x|)dx$ with $v$ continuous and
$$(v/\om'_{1})(|x|),$$ 
an increasing function of $|x|$, $x\in K_{r}$. Then 
there exists $r^{*}\in(r,1)$ and a positive measure, denoted $T(\sigma)$, with support $K_{r^{*}}$, 
such that
$$
U^{T(\sigma)}(x)=U^{\sigma}(x)+C,\quad x\in K_{r^{*}},\qquad\text{and }\quad m(T(\sigma))=m(\sigma),
$$
where $C$ is some constant.
\end{proposition}
The proof is based on building a sequence of signed measures $\sigma_{k}$, starting with $\sigma$, by performing iterative balayages of their negative components onto their positive components, as defined by the map $J$ in (\ref{def-J}). This idea was already used in \cite[Theorem 2]{KD} for the logarithmic kernel. Here, we also need to cope with the mass gain induced by the balayage of negative measures, which is one of the specifities of the Riesz theory. Hence, at each step of the procedure, we modify $\sigma_{k}$ in a correct way before applying the map $J$.
\begin{proof}
The sequence of measures $\sigma_{k}$, $k\geq0$, is built recursively, in the following way
$$
\sigma_{0}=\sigma,\quad\sigma_{k+1}=J(\sigma_{k}-c_{\sigma_{k}}\om_{1}),\qquad k\geq0,
$$
where $c_{\sigma_{k}}$ is given by 2) of Lemma \ref{lem-J}. Hence, each measure $\sigma_{k}$, $k\geq0$, has mass $m(\sigma)$.

From the continuity of $v$
and 2) of Proposition \ref{bal-abs-cont},
it follows that $\sigma_{k}=v_{k}(t)dt$ is also absolutely continuous, radial, with a continuous density $v_{k}$. 
Note that, for each $k$, $\sigma_{k+1}^{-}$ is a nontrivial measure because the density of 
$(\sigma_{k}-c_{\sigma_{k}}\om_{1})^{+}$ vanishes at $r_{k+1}$ while the density of the balayage of $(\sigma_{k}-c_{\sigma_{k}}\om_{1})^{-}$ tends to $-\infty$ on the right of $r_{k+1}$. The measure $\sigma_{k+1}^{+}$ is nontrivial as well because $m(\sigma_{k})=m(\sigma)>0$. Thus, we can use Lemma \ref{lem-J} repeatedly, where we also observe that $(\sigma_{k}-c_{\sigma_{k}}\om_{1})'/\om_{1}'=\sigma_{k}'/\om_{1}'-c_{\sigma_{k}}$ is increasing. Consequently, we obtain a non-decreasing sequence 
$r=r_{0}< r_{1}< r_{2}<\cdots\leq1$ such that 
$$\supp\sigma_{k}=K_{r_{k}},\quad
\sigma_{k}\leq\sigma_{k-1}\text{ on }K_{r_{k}},\quad
m(\sigma_{k})=m(\sigma_{k-1})=m(\sigma),
$$
and $(v_{k}/\om_{1}')(|t|)\text{ increases on }K_{r_{k}}$. 

Next, from the definition of the sequence $\{\sigma_{k}\}_{k}$, it follows that
$$
\sigma_{k}=Bal(\sigma-\sum_{j=0}^{k-1}c_{\sigma_{j}}\om_{1},K_{r_{k}}).
$$
Thus, applying (\ref{mass-bal}),
\begin{equation}\label{eq-for-c}
m(\sigma_{k})=\capa(K_{r_{k}})\left(\int U^{\sigma}d\om_{K_{r_{k}}}-
\sum_{j=0}^{k-1}c_{\sigma_{j}}\int U^{\om_{1}}d\om_{K_{r_{k}}}\right),
\end{equation}
which implies, together with the continuity of $U^{\sigma}$, 
that 
$$\frac{m(\sigma_{k})}{\capa(K_{r_{k}})}\leq\sup_{I} U^{\sigma}-\capa(I)\sum_{j=0}^{k-1}c_{\sigma_{j}},
$$
showing that the series $\sum_{j=0}^{\infty} c_{\sigma_{j}}$ converges, since, otherwise, we would have, for $k$ large, $m(\sigma_{k})=m(\sigma)<0$, a contradiction. 
Let
\begin{equation}\label{3-def}
r^{*}=\lim_{k}r_{k}\leq1,\qquad c^{*}=\sum_{j=0}^{\infty} c_{\sigma_{j}}<\infty,\qquad\sigma^{*}=Bal(\sigma-c^{*}\om_{1},K_{r^{*}}).
\end{equation}
Note that 
\begin{align*}
I(\sigma^{*}) & \leq I(Bal(\sigma^{-},K_{r^{*}}))+I(Bal(\sigma^{+},K_{r^{*}}))+I(Bal(c^{*}\om_{1},K_{r^{*}})) \\[5pt]
& \leq I(\sigma^{-})+I(\sigma^{+})+I(c^{*}\om_{1})<\infty,
\end{align*}
and thus $r^{*}<1$.
From Lemma \ref{lem-iba} we get that
$$
\sigma_{k}\to\sigma^{*}\quad\text{weak-*, \quad as $k\to\infty$}.
$$
Since the restrictions of $\sigma_{k}$ to $K_{r^{*}}$ are positive measures for all $k$, and 
$\sigma^{*}$ does not charge the endpoints of $K_{r^{*}}$, \cite[Theorem 0.5']{Land} applies, showing that
$$
\sigma^{*}
=\lim_{k}\sigma_{k}|_{K_{r^{*}}},
$$
which is a positive measure. It follows from the property of balayage that 
$$U^{\sigma^{*}}(x)=U^{\sigma}(x)-c^{*}W(I),\quad x\in K_{r^{*}},\qquad
m(\sigma^{*})=\lim_{k}m(\sigma_{k})=m(\sigma).
$$ 
Thus we can set $T(\sigma)=\sigma^{*}$, which satisfies the stated properties. 
\end{proof}
\begin{remark}\label{Rem-v*}
Let $v^{*}$ on $K_{r^{*}}$ be the pointwise limit of the decreasing sequence of positive functions $v_{k}$. Then, by the Lebesgue dominated convergence theorem, $\sigma_{k}\to v^{*}(t)dt$, weak-*, and thus $T(\sigma)=v^{*}(t)dt$. In particular, we have :\\
(i) the quotient $v^{*}/\om_{1}'$ is an increasing function of the radius $r\in[r^{*},1)$,\\
(ii) the function $v^{*}$ is continuous on $(r^{*},1)$, which follows from the continuity of $v$, the expression of $T(\sigma)$ as a balayage (last equality in (\ref{3-def})), and 2) of Proposition \ref{bal-abs-cont}.
\end{remark}
We aim at applying Proposition \ref{IBA} to the case where the initial measure $\sigma$ is a signed equilibrium measure for an external field $Q$. For that, we need a final result which improves Lemma \ref{inclusion}.
\begin{lemma}\label{Lem-KD-impr}
Let $\sigma=\eta_{Q,K_{r}}$, $0\leq r<1$, be the signed equilibrium measure for an external field $Q$ on $K_{r}$, satisfying the assumptions in Lemma \ref{lem-J}. Then
\begin{equation}\label{KD-impr}
\om_{Q,K_{r}}\leq(\sigma-c_{\sigma}\om_{I})^{+},
\end{equation}
where $c_{\sigma}$ is the constant from item 2) of Lemma \ref{lem-J}.
\end{lemma}
Note that Lemma \ref{inclusion} only asserts that $\om_{Q,K_{r}}\leq\sigma^{+}$ which is weaker than (\ref{KD-impr}).
\begin{proof}
From \eqref{defsigned} and \eqref{Frostman1}--\eqref{Frostman2} we have that
\begin{align}\label{PAMS}
U^{(\sigma-c_{\sigma}\om_{I})^{+}}(x) & \leq U^{(\sigma-c_{\sigma}\om_{I})^{-} + \om_{Q,K_{r}}}(x) 
+ C-F_Q-c_{\sigma}W(I),\quad\text{q.e. on } K_{r},
\\\label{PAMS2}
U^{(\sigma-c_{\sigma}\om_{I})^{+}}(x)& = U^{(\sigma-c_{\sigma}\om_{I})^{-} + \om_{Q,K_{r}}}(x) 
+ C-F_Q-c_{\sigma}W(I),\quad\text{q.e. on }S_{{Q,K_{r}}},
\end{align}
where $C = C_{Q,K_{r}}$ is the equilibrium constant for the signed equilibrium measure $\sigma=\eta_{Q,K_{r}}$.
Let $K_{\tilde r}:=\supp J(\sigma-c_{\sigma}\om_{I})$. From the property of balayage, we deduce from (\ref{PAMS}) that
\begin{align}\label{PAMS3}
U^{J(\sigma-c_{\sigma}\om_{I})}(x) & \leq U^{\om_{Q,K_{r}}}(x) 
+ C-F_Q-c_{\sigma}W(I),\quad\text{q.e. on } K_{\tilde r},
\end{align}
Integrating (\ref{PAMS3}) with respect to $\om_{K_{\tilde r}}$ and applying Tonelli's theorem, we derive that
$$
m(J(\sigma-c_{\sigma}\om_{I}))W(K_{\tilde r})\leq m(\om_{Q,K_{r}})W(K_{\tilde r})
+ C-F_Q-c_{\sigma}W(I).
$$
Since $m(J(\sigma-c_{\sigma}\om_{I}))=m(\sigma)=1$ and $m(\om_{Q,K_{r}})=1$, we obtain that
$0\leq C-F_Q-c_{\sigma}W(I)$.
Thus, in view of (\ref{PAMS})-(\ref{PAMS2}), the assumptions in \cite[Theorem 3.2]{DOSW} with 
$$\mu=\om_{Q,K_{r}}+(\sigma-c_{\sigma}\om_{I})^{-},\qquad\nu=(\sigma-c_{\sigma}\om_{I})^{+},\qquad C=C-F_Q-c_{\sigma}W(I),$$ 
are satisfied, which shows that
$\om_{Q,K_{r}}+(\sigma-c_{\sigma}\om_{I})^{-}\leq(\sigma-c_{\sigma}\om_{I})^{+}$ on $S_{{Q,K_{r}}}$ and consequently $\om_{Q,K_{r}}\leq(\sigma-c_{\sigma}\om_{I})^{+}$, as was to be proved.
\end{proof}
We are now in a position to apply the previous results to get our main theorem.
\begin{theorem}\label{KD2}
Let $\eta_{Q,1}$ be the signed equilibrium measure on $I=[-1,1]$ corresponding to an external field $Q$. Assume that $\eta_{Q,1}(x)=v(|x|)dx$ is an absolutely continuous, radial measure, with a continuous density $v(|x|)$ such that $(v/\om'_{1})(|x|)$ increases with $|x|$, $x\in I$. Then it holds that :\\
(i) the support $S_{Q,I}$ of the equilibrium measure $\om_{Q,1}$ is the union of two symmetric intervals, namely $S_{Q,I}=K_{r^{*}}$, for some $0\leq r^{*}<1$.\\
(ii) As a function of the radius $r$, the density of $\om_{Q,1}=\eta_{Q,K_{r^{*}}}$ is a continuous,  increasing function of $r\in[r^{*},1]$ which grows, up to a positive multiplicative constant, at least like $(1-r)^{-\alpha/2}$ as $r\uparrow1$.
\\
(iii) Assume $Q$ and $v$ are of class $C^{1}$ on $I$. Then, as a function of $r$, the density of $\om_{Q,1}$ vanishes at $r^{*}$.
\end{theorem}
\begin{remark}
The vanishing of the density of the weighted equilibrium measure at the inner (soft) endpoints agrees with the generic behavior of the density at the soft endpoints established 
in \cite{KuML} for logarithmic potentials.
\end{remark}
\begin{proof}
(i) It suffices to apply Proposition \ref{IBA} with the initial measure $\sigma=\eta_{Q,I}$. We then get that $T(\sigma)=\om_{Q,K_{r^{*}}}$. Moreover applying Lemma \ref{Lem-KD-impr} on each subset $K_{r_{k}}$ (where we keep the same notations as in the proof of Proposition \ref{IBA}), we derive that
$$
\forall k\geq0,\qquad\supp\om_{Q,I}\subset K_{r_{k}},
$$
hence $\supp\om_{Q,I}\subset K_{r^{*}}$. It follows that
$\om_{Q,I}=\om_{Q,K_{r^{*}}}$
and $S_{Q,I}=K_{r^{*}}$.
\\
(ii) The fact that the density of $\om_{Q,1}'$ is an increasing function of $r\in[r^{*},1]$ follows from item  (i) of Remark \ref{Rem-v*} and the fact that $\om_{1}'$ is also increasing with $r$. From item (ii)
 of Remark \ref{Rem-v*} we know that $\om_{Q,1}'$ is continuous on $(r^{*},1)$. Then, we can extend $\om_{Q,1}'$ to a continuous function on $[r^{*},1)$ by setting $\om_{Q,1}'(r^{*})=\lim_{r\downarrow r^{*}}\om_{Q,1}'(r)$, which exists.

The fact that $\om_{Q,1}'/\om_{1}'$ is increasing also implies the stated estimate on the growth of $\om_{Q,1}'$ near 1.
\\
(iii) From the Frostman inequalities (\ref{Frostman1})--(\ref{Frostman2}), one gets
$$
U^{\om_{Q,1}}(x)+Q(x)= F_{Q},\quad x\in K_{r^{*}}\quad\implies\quad
(U^{\om_{Q,1}})'(x)=-Q'(x),\quad x\in K_{r^{*}}.
$$
Since, by assumption, $Q'$ is a function continuous on $I$, it implies that the derivative 
\begin{equation}\label{der-pot}
(U^{\om_{Q,1}})'(x)=\frac{d}{dx}\int_{K_{r^{*}}}\frac{\om_{Q,1}'(t)dt}{|x-t|^{s}}
\end{equation}
is also bounded on $K_{r^{*}}$. For some $r^{*}<a<1$ and $r^{*}<x<a$, we split the previous expression into four pieces,
$$
\frac{d}{dx}\int_{-1}^{-r^{*}}\frac{\om_{Q,1}'(t)dt}{(x-t)^{s}}
+\frac{d}{dx}\int_{r^{*}}^{x}\frac{\om_{Q,1}'(t)dt}{(x-t)^{s}}
+\frac{d}{dx}\int_{x}^{a}\frac{\om_{Q,1}'(t)dt}{(t-x)^{s}}
+\frac{d}{dx}\int_{a}^{1}\frac{\om_{Q,1}'(t)dt}{(t-x)^{s}}.
$$
Observe that, as $r^{*}<x<a$, the first and fourth derivatives remain bounded. Moreover, applying Lemma \ref{Samko}, where we note that $\om_{Q,1}'$ is absolutely continuous on $[r^{*},a]$ (being continuous, increasing on $[r^{*},a]$, and of class $C^{1}$ on $(r^{*},a]$), we get
\begin{align*}
\frac{d}{dx}\int_{r^{*}}^{x}\frac{\om_{Q,1}'(t)dt}{(x-t)^{s}} & =\frac{\om_{Q,1}'(r^{*})}{(x-r^{*})^{s}}
+\int_{r^{*}}^{x}\frac{\om_{Q,1}''(t)dt}{(x-t)^{s}}\geq\frac{\om_{Q,1}'(r^{*})}{(x-r^{*})^{s}},
\\[5pt]
\frac{d}{dx}\int_{x}^{a}\frac{\om_{Q,1}'(t)dt}{(t-x)^{s}} & =-\frac{\om_{Q,1}'(a)}{(a-x)^{s}}
+\int_{x}^{a}\frac{\om_{Q,1}''(t)dt}{(t-x)^{s}}\geq-\frac{\om_{Q,1}'(a)}{(a-x)^{s}}.
\end{align*}
As $x\downarrow r^{*}$, the expression (\ref{der-pot}) remains bounded, while the lower bound in the second inequality has a finite limit. In view of the first inequality, we derive that $\om_{Q,1}'(r^{*})=0$.
\end{proof}
We now choose $Q$ to be the particular external field we are interested in, namely the one created by a positive point charge over $I$, that is $Q=\gamma U^{\delta_{y}}$. We apply Theorem \ref{KD2} to the corresponding signed equilibrium measure 
\begin{equation}\label{exp-eta}
\eta_{Q,I}=-\gamma Bal(\delta_{y},I)+(1+\gamma m_{1})\om_{I},
\end{equation}
recall (\ref{sgn-meas}), which is possible, since, by Lemma \ref{lem-quo-incr}, we know that $(\eta'_{Q,I}/\om'_{1})(|x|)$ increases as a function of $|x|$. We get
\begin{corollary}\label{supp-shell}
The following holds:\\
(i) The equilibrium measure $\om_{Q,I}$ on the segment $I$ in the external field (\ref{ext-field}) of a repulsive charge $\gamma>\gamma_{+}$ has a support $S_{Q,I}$ which is the union $K_{r^{*}}$ of two symmetric intervals, for some $0<r^{*}<1$. \\
(ii) The density of $\om_{Q,1}=\eta_{Q,K_{r^{*}}}$ behaves, up to a positive multiplicative constant, like $(1-|x|)^{-\alpha/2}$ as $|x|\uparrow1$, and vanishes at $|x|=r^{*}$.
\end{corollary}
\begin{proof}
We only need to prove that the growth of $\om_{Q,1}'$ is not larger than $(1-|x|)^{-\alpha/2}$ as $|x|\uparrow1$. But it follows from (\ref{signeddomin}) and (\ref{exp-eta}) that 
$\om_{Q,1}\leq(\eta_{Q,1})^{+}\leq(1+\gamma m_{1})\om_{I}$, and
$\om_{I}$ has the stated behavior as $|x|\uparrow1$, recall Theorem \ref{Wal}, (iii). 
\end{proof}
\begin{remark}\label{Rem-comp-rad}
How to compute the inner radius $r^{*}$ of the support ? Assume we are able to compute, for a general set 
$$K=[a_{1},a_{2}]\cup[a_{3},a_{4}],$$ 
the constants $M_{i}(K)$, $1\leq i\leq4$, such that
$$
\om_{K}'(x)\to\frac{M_{i}(K)}{|a_{i}-x|^{\alpha/2}},\quad\text{ as }|x|\to a_{i},~ 1\leq i\leq4.
$$
Then, by Lemma \ref{Bal-eq}, for $y\notin K$, we also know the constants $N_{i}(y,K)$ such that
$$
Bal(\delta_{y},K)'(x)\to\frac{N_{i}(y,K)}{|a_{i}-x|^{\alpha/2}},\quad\text{ as }|x|\to a_{i},~ 1\leq i\leq4,
$$
and, from the superposition principle, for a measure $\nu$ supported outside of $K$, we can also derive the constants $N_{i}(\nu,K)$ such that
$$
Bal(\nu,K)'(x)\to\frac{N_{i}(\nu,K)}{|a_{i}-x|^{\alpha/2}},\quad\text{ as }|x|\to a_{i},~ 1\leq i\leq4,
$$
(this would need more details to be justified).
Now, let us choose $K=K_{r^{*}}=[-1,-r^{*}]\cup[r^{*},1]$. Recall that $\om_{Q,1}=Bal(\eta_{Q,1}-c^{*}\om_{1},K_{r^{*}})$ and observe that (\ref{mass-bal}) gives the relation,
$$
1=\capa(K_{r^{*}})\left(C_{Q,I}-c^{*}W(I)\right),
$$
where $C_{Q,I}$ is the signed equilibrium constant in (\ref{defsigned}).
Then the fact that $(\sigma^{*})'(r^{*})$ is a finite number implies that $r^{*}$ must be a solution to the following equation with respect to~$r$,
$$
N_{2}(\eta_{Q,I}|_{[-r,r]},K_{r})=c^{*}N_{2}(\om_{1}|_{[-r,r]},K_{r}),
$$
where, by symmetry, $N_{2}$ can be replaced with $N_{3}$. Assuming moreover that the capacity of a set $K_{r}$ could be computed, then the above equation would allow one to determine $r^{*}$.
\end{remark}
\noindent
{\bf Funding.} The research of the Peter Dragnev is supported, in part, by the Lilly Endowment.

\obeylines
\texttt{P.~Dragnev (dragnevp@pfw.edu)
Department of Mathematical Sciences, Purdue University Fort Wayne, 
Ft.\ Wayne, IN 46805, USA.
\medskip
R.~Orive (rorive@ull.es)
Departmento de An\'{a}lisis Matem\'{a}tico, Universidad de La Laguna, and 
Instituto de Matem\'aticas y Aplicaciones de la ULL (IMAULL), 
38200 La Laguna (Tenerife), SPAIN.
\medskip
E.~B.~Saff (edward.b.saff@vanderbilt.edu)
Center for Constructive Approximation, Vanderbilt University, 
Department of Mathematics, Nashville, TN  37240, USA.
\medskip
F.~Wielonsky (franck.wielonsky@univ-amu.fr)
Laboratoire I2M, UMR CNRS 7373, Universit\'e Aix-Marseille, 
F-13453 Marseille Cedex 20, FRANCE.
}
\end{document}